\documentclass[12pt]{article}
\usepackage{amsmath,amsthm,amssymb,amsfonts}
\newcommand{\Mod}[1]{\ (\mathrm{mod}\ #1)}
\usepackage[]{natbib} 
\usepackage[pagebackref=false,
            linktoc=page,
            colorlinks=true,
            linkcolor=blue,
            citecolor=blue]{hyperref}
\usepackage{authblk}
\usepackage{pstool,psfrag}
\usepackage{bbm}
\usepackage{bm}
\usepackage{epsfig}
\usepackage{psfrag}
\usepackage{enumerate}
\usepackage{booktabs}
\usepackage{slashbox}
\usepackage{caption}
\usepackage[utf8]{inputenc}
\numberwithin{equation} {section}
\usepackage{graphicx}
\usepackage{mathtools}

\DeclarePairedDelimiter\floor{\lfloor}{\rfloor}
\usepackage{geometry}
\usepackage{listings}
\usepackage{tgtermes}
\usepackage{tgheros}
\usepackage{tgcursor}
\usepackage{color}
\usepackage{fullpage}
\newtheorem{Definition}{Definition}

\numberwithin{Theorem}{section}
\numberwithin{Definition}{section}
\numberwithin{Lemma}{section}
\numberwithin{Algorithm}{section}
\numberwithin{equation}{section}

\theoremstyle{plain}
\newtheorem{theorem}{Theorem}[section]

\newtheorem{corollary}[theorem]{Corollary}
\newtheorem{lemma}{Lemma}[section]

\theoremstyle{definition}

\theoremstyle{remark}

\author[$\dagger$\footnote{Corresponding author}]{\textbf{Graeme Auld} }
\author[$\dagger$]{\textbf{Ioannis Papastathopoulos}}
\affil[$\dagger$]{\small School of Mathematics and Maxwell Institute,
  University of Edinburgh, Edinburgh, EH9 3FD}
\affil[$$]{\small
  G.R.Auld@sms.ed.ac.uk
  $\quad$ i.papastathopoulos@ed.ac.uk} 
\date{}
\begin{document}
\title{\textbf{Extremal clustering in non-stationary random sequences}}

\maketitle

\begin{abstract}
It is well known that the distribution of extreme values of strictly stationary sequences differ from those of independent and identically distributed sequences in that extremal clustering
may occur. Here we consider non-stationary but identically distributed sequences of random variables subject to suitable long range dependence restrictions.  
We find that the limiting distribution of appropriately normalized sample maxima depends on a parameter that measures the average extremal clustering of the sequence. Based on this new representation we derive the asymptotic distribution for the time between consecutive extreme observations and construct moment and likelihood based estimators for measures of extremal clustering.\    
We specialize our results to random sequences with periodic dependence structure.
\end{abstract} {\bf Keywords:} Clustering of extremes, extremal index,
interexceedance times, intervals estimator, non-stationary sequences, periodic processes.
\section{Introduction} \label{Intro}
\label{sec:intro}

Extreme value theory for strictly stationary sequences has been extensively studied, initiated in the works of \cite{watson54}, \cite{berman64}, \cite{loynes65}, 
and continued by \cite{lead74, lead83} and \cite{O'Brien87} amongst others. 
One of the key findings in this line of research is that unlike in independent and identically distributed sequences where extreme values tend to occur in isolation, stationary sequences possess an intrinsic potential for clustering of extremes, i.e.,\ several successive or close extreme values may be observed.  Understanding the extremal clustering characteristics of a stochastic process is critical in many applications where a cluster of extreme values may have serious consequences.  For example, if a sequence consists of daily temperatures at some fixed location then a cluster of extremes may correspond to a heatwave.  

The extent to which extremal clustering may occur is naturally measured, for strictly stationary sequences, by a parameter known as the extremal index.
Let $\{X_n\}_{n=1}^{\infty}$ be a sequence of random variables with common marginal distribution function $F$, and let $\bar{F}=1-F$ and
$M_n = \text{max}\{X_1,\ldots,X_n\}$. Also, let $\{x_n\}_{n=1}^{\infty}$ be a sequence of real numbers that we may informally think of as thresholds or levels.  
In the special case that $X_i$ and $X_j$ are independent, $i\neq j$, then a necessary and sufficient condition for $\mathbb{P}(M_n \leq x_n)$ to converge to a limit in $(0,1)$ as $n\to\infty$ is that
$n\bar{F}(x_n) \to \tau > 0,$ in which case $\mathbb{P}(M_n \leq x_n) \to e^{-\tau}$ \citep[Theorem 1.5.1]{leadling83}.  More generally, if $\{X_n\}_{n=1}^{\infty}$
is a strictly stationary sequence, then $n\bar{F}(x_n) \to \tau$ is not sufficient to ensure the convergence of $\mathbb{P}(M_n \leq x_n)$. However, in most cases of practical interest,
provided that a suitable long range dependence restriction is satisfied, such as condition $D$ of \cite{lead74}, one has $\mathbb{P}(M_n \leq x_n) \to e^{-\theta\tau}$
where $\theta \in [0,1]$ is the extremal index. \cite{lead83} showed that exceedances of the level $x_n$ occur in clusters with the limiting mean cluster size being equal to $\theta^{-1}$, and \cite{hsing87} showed that distinct clusters may be considered independent in the limit.     

Another characterization of $\theta$ that links it to the extremal clustering properties of a strictly stationary sequence can be found in \cite{O'Brien87}. Defining $M_{j,k} = \text{max}\{X_i  : j+1\leq i \leq k \}$, 
it was shown that the distribution function of $M_n$ satisfies
\begin{equation}
\mathbb{P}(M_n \leq x_n) - F(x_n)^{n\theta_n} \to 0, \quad \text{as } n\to \infty,  \label{eq1.1}
\end{equation}
where
\begin{equation}
\theta_n = \mathbb{P}(M_{1,p_n} \leq x_n \mid X_1 > x_n),    \label{eq1.2}
\end{equation}
 for some $p_n = o(n)$, and provided the limit exists, $\theta_n \to \theta$ as $n \to \infty$.  This result illustrates that smaller values of $\theta$ are indicative of a larger degree of extremal clustering, since the conditional probability in (\ref{eq1.2}) is small when an exceedance of a large threshold is likely to soon be followed by another exceedance. 

Early attempts at estimating $\theta$ were based on associating $\theta^{-1}$ with the limiting mean cluster size.  Different methods for identifying clusters gave rise to  different estimators, well known examples being the runs and blocks estimators \citep{smithweiss94}.  For the runs estimator, a cluster is identified as being initialized when a large threshold is exceeded and ends when a fixed number, known as the run length, of non-exceedances occur. The extremal index is then estimated by the ratio of the number of identified clusters to the total number of exceedances.  A difficulty that arises when using this estimator is its sensitivity to the choice of run length \citep{hsing91}.  

The problem of cluster identification was studied by \cite{ferrseg03} who considered the distribution of the time between two exceedances of a large threshold.  They found that the limiting distribution of appropriately normalized interexceedance times converges to a distribution that is indexed by $\theta$. In particular, for a given threshold $u \in \mathbb{R}$, they define the random variable $T(u) = \text{min}\{n \geq 1 : X_{n+1} > u \mid X_1 > u \}$, and found that as $n\to \infty,$ $\bar{F}(x_n)T(x_n)$ converges in distribution to a mixture of a point mass at zero and an exponential distribution with mean $\theta^{-1}$. 
Thus, by computing theoretical moments of this limiting distribution and comparing them with their empirical counterparts, they construct their so-called intervals estimator.

Motivated by the fact that many real world processes are non-stationary, in this paper we investigate the effect of non-stationarity on extremal 
clustering.\ Previous statistical works that consider extremal clustering in non-stationary sequences include \cite{suv07}, who used the likelihood function introduced by \cite{ferrseg03} for the extremal index together with smoothing methods to capture non-stationarity in a time series of temperature measurements. In a similar application, \cite{colestawnsmith97} used a Markov model together with simulation techniques to estimate the extremal index within different months.

An early work that developed extreme value theory for non-stationary sequences with a common marginal distribution is \cite{husl83}, which focused on the asymptotic distribution of the sample maxima but did not consider extremal clustering. 
\cite{husl86} considered the more general case where the margins may differ and also discussed the difficulty of defining the extremal index for general non-stationary sequences.

Here, we consider a sequence of random variables $\{X_n\}_{n=1}^{\infty}$ with common marginal distribution function $F$, but do not assume stationarity in either the weak or strict sense. As we assume common margins, non-stationarity may arise through changes in the dependence structure. We show, under  assumptions similar to  \cite{O'Brien87}, that
\begin{equation}
\mathbb{P}(M_n \leq x_n) - F(x_n)^{n\gamma_n} \to 0, \quad  \text{as } n\to \infty,    \label{eq1.3}
\end{equation}
where 
\begin{equation}
\gamma_n = \frac{1}{n} \sum_{j=1}^{n} \mathbb{P}(M_{j,j+p_n} \leq x_n \mid X_j> x_n).   \label{eq1.4} 
\end{equation}
Thus, we find that the limiting distribution of the sample maximum at large thresholds is characterized by a parameter $\gamma  = \lim_{n \to \infty} \gamma_n$, provided the limit exists, which by analogy with equation (\ref{eq1.2}), may be regarded as the average of local extremal indices. In this paper we develop methods for estimating these local extremal indices by adapting the methods of \cite{ferrseg03} for the extremal index to our non-stationary setting. In the special case that the sequence is stationary, so that all terms in the summation (\ref{eq1.4}) are equal, the formula for $\gamma_n$ reduces to $\theta_n$ in (\ref{eq1.2}). 

The structure of the paper is as follows.\ Section \ref{Sec2} defines the notation and assumed mixing condition used throughout the paper and states the main theoretical results regarding the asymptotic distribution of the sample maxima and normalized interexceedance times.  Section \ref{Sec3} discusses approaches to parameter estimation using the result from Section \ref{Sec2} on the distribution of the 
interexceedance times.\  Section \ref{Sec4} considers the estimation problem for two simple non-stationary Markov sequences with periodic dependence structures and Section \ref{SecProofs} gives the proofs of the main theoretical results. 

\section{Theoretical results}    \label{Sec2}
 
\subsection{Notation, definitions and preliminary results}   \label{Sec2.1} 

Throughout the paper, when not explicitly stated otherwise, all limits should be interpreted as ``as $n\to\infty$''.
We assume that all random variables in the sequence $\{X_n\}_{n=1}^{\infty}$ have common marginal distribution $F$ with upper endpoint $x_F = \text{sup}\{x\in \mathbb{R} : F(x) < 1 \}$, though we do not assume stationarity.  
In addition to the definitions for $M_n$ and $M_{j,k}$ given in the Section \ref{Intro}, we define $M(A) = \text{max}\{ X_i : i\in A \}$ where $A$ is an arbitrary set of positive integers, and write $ | A\, |$ for the number of elements in $A$.
We also refer to a set of consecutive integers as an interval.  
If $I_1$ and $I_2$ are two intervals, we say that $I_1$ and $I_2$ are separated by $q$ if min($I_2$) - max($I_1$) = $q+1$ 
or min($I_1$) - max($I_2$) = $q+1$, i.e., there are $q$ intermediate values between $I_1$ and $I_2$.
The set $\{1,2,3,\ldots\}$ is denoted by $\mathbb{N}$.   
Equality in distribution of two random variables $X$ and $Y$ is denoted by  $ X \overset{D}{=} Y.$

We assume that the sequence $\{X_n\}_{n=1}^{\infty}$ satisfies the asymptotic independence of maxima (AIM) mixing condition of \cite{O'Brien87} which restricts long range dependence.
\begin{Definition} \label{Defn1} \normalfont The sequence $\{X_n\}_{n=1}^{\infty}$ is said to satisfy the asymptotic independence of maxima condition relative to the sequence $x_n$ of real numbers, abbreviated to ``$\{X_n\}_{n=1}^{\infty}$ satisfies \text{AIM}($x_n$)'', if there exists a sequence $q_n$ of positive integers with $q_n = o(n)$ such that for any two intervals $I_1 = \{i_1,\ldots, i_j \}$ and $I_2 = \{i_j + q_n + 1,\ldots, i_j + q_n + k \}$ separated by $q_n,$ we have 
\begin{equation}
\alpha_n = \text{max}\, \mid\mathbb{P}\big(M(I_1 \cup I_2) \leq x_n \big) - \mathbb{P}\big(M(I_1) \leq x_n\big)\mathbb{P}\big(M(I_2) \leq x_n\big)\mid \rightarrow 0,  \label{AIMalpha}
\end{equation}
where the maximum is taken over all positive integers $i_1, i_j$ and $k$ such that $|I_1\,| \geq q_n$, $|I_2\,| \geq q_n$ and $i_j + q_n + k  \leq n$.
\end{Definition}
Definition \ref{Defn1} states a slightly weaker condition than the widely used $D(x_n$) condition \citep{lead83} in that only certain intervals $I_1$ and $I_2$ need to be considered in
(\ref{AIMalpha}) rather than arbitrary sets of integers, so that all examples in the literature of sequences satisfying 
$D(x_n$) also satisfy AIM($x_n$).  For example, stationary Gaussian sequences with autocorrelation function $\rho_n$ satisfying Berman's condition, 
$\rho_n\text{log}\,n\to 0$ \citep{berman64}, satisfy AIM($x_n$) for any sequence $x_n$ such that $n\bar{F}(x_n)$ is bounded and any $q_n=o(n)$ \citep[Lemma 4.4.1]{leadling83}. 
The analogous result for non-stationary Gaussian sequences is given in \cite{husl83}, where Berman's condition is replaced by $r_n\text{log}\,n \to 0$ with $r_n = \text{sup}\{|\rho(i,j)| : |i-j| \geq n\}$ and $\rho(i,j)$ the correlation between $X_i$ and $X_j$.

\cite{O'Brien87} showed that if $\{X_n\}_{n=1}^{\infty}$ is a stationary positive Harris Markov sequence with separable state space $S$ and $f:S\to \mathbb{R}$ is a measurable function then the sequence $Y_n = f(X_n)$ satisfies AIM($x_n$) for any $x_n$ and $q_n=o(n)$ with $q_n \to \infty$.

We note that Definition \ref{Defn1} states a property of the dependence structure of the sequence $\{X_n\}_{n=1}^{\infty}$, with the specific marginal distributions playing essentially no role.  In particular, if $\{X_n\}_{n=1}^{\infty}$ satisfies AIM($x_n$) and $g:\mathbb{R}\to \mathbb{R}$ is a monotone increasing function then $Y_n = g(X_n)$ satisfies AIM($g(x_n)$) with the same $q_n$.

The assumption that $\{X_n\}_{n=1}^{\infty}$ satisfies AIM($x_n$) ensures the approximate independence of the block maxima of two sufficiently separated blocks. Lemma \ref{FirstLemma} below provides an upper bound for the degree of dependence of $k$
block maxima for suitably separated blocks and will be useful in Section \ref{Sec2.2} when the limiting behaviour of $\mathbb{P}(M_n \leq x_n)$ is considered. 
\begin{lemma} \label{FirstLemma}
\normalfont  \label{SepLem} Let $\{X_n\}_{n=1}^{\infty}$ satisfy AIM($x_n$) and let $I_1, I_2,\ldots,I_k$ be distinct subintervals of $\{1,2,\ldots, n\}$ where $k\geq 2$ and $|I_i| \geq q_n$, $1\leq i \leq k$.  Suppose that $I_i$ and $I_{i+1}$ are separated by $q_n$ for  $1\leq i \leq k-1$.  Then
\begin{equation}
\big|\mathbb{P}(M(\cup_{i=1}^{k}I_i) \leq x_n) - \prod_{i=1}^{k} \mathbb{P}(M(I_i) \leq x_n) \big| \leq (k-1)\alpha_n + 2(k-2)q_n\bar{F}(x_n).  \label{SepLemEqn}
\end{equation} 
\end{lemma}

\subsection{Asymptotic distribution of $M_n$}   \label{Sec2.2}

In this section we investigate the limiting behaviour of $\mathbb{P}(M_n \leq x_n)$, with the main result being Theorem \ref{MainThm}.  In addition to assuming that $\{X_n\}_{n=1}^{\infty}$ satisfies AIM($x_n$), we will assume that the rate of growth of the sequence $x_n$ is controlled via 
\begin{equation}
n\bar{F}(x_n) \to \tau >0.  \label{ExpExc}  % expected number of exceedances 
\end{equation}
In the case of continuous marginal distributions, (\ref{ExpExc}) is immediately satisfied by $x_n = F^{-1}(1-\tau/n)$. 
More generally, Theorem 1.7.13 of \cite{leadling83} guarantees the existence of a sequence $x_n$ satisfying
(\ref{ExpExc}) when $F$ is in the domain of attraction of any of the three classical extreme value distributions \citep[Section 1.2]{dehaanbook}. 

We use the standard technique of block-clipping, see for example Section 10.2.1 in \cite{beiretal04}, to split the interval $\{1,2,\ldots, n\}$ into subintervals, or blocks, of alternating large and small lengths.  Specifically, for sequences $p_n$ and $q_n$ such that $q_n = o(p_n)$ and $p_n = o(n)$ we define
\begin{align}
A_i  &= \big\{(i-1)(p_n+q_n)+1,\ldots, ip_n + (i-1)q_n \,\big \}   \label{A_i}  \\
A_i^* & = \big\{ ip_n + (i-1)q_n + 1, \ldots, i(p_n + q_n) \big\},  \notag
\end{align}
for $i = 1,2,\ldots r_n$, where $r_n =  \floor*{n/(p_n + q_n)}$. 

If we take the sequence $q_n$ appearing in the construction of the blocks $A_i$ and $A_i^*$ to be the same as that in Definition \ref{Defn1}, then Lemma 
\ref{SepLem} bounds the degree of dependence of the collection of random variables $\{M(A_i)\}_{i=1}^{r_n}$, and this allows us to prove 
Lemma \ref{lemOB} below which modifies 
Lemma 3.1 from  \cite{O'Brien87} to allow for non-stationarity.

\begin{lemma} \normalfont  \label{lemOB}
Let $\{X_n\}_{n=1}^{\infty}$ satisfy AIM$(x_n)$ and let the sequence $p_n$ be such that 
\begin{equation}
p_n = o(n), \quad n\alpha_n = o(p_n) \quad \text{and} \quad q_n = o(p_n).   \label{lemOBeq1}
\end{equation}
Then if (\ref{ExpExc}) holds, we have
\begin{equation} 
\mathbb{P}(M_n \leq x_n) - \prod_{i=1}^{r_n} \mathbb{P}(M(A_i) \leq x_n) \rightarrow 0, \label{eq2.5}
\end{equation}
where the intervals $\{A_i\}_{i=1}^{r_n}$ are as in (\ref{A_i}). 
\end{lemma}
\sloppy \textbf{Remarks.} Equation (\ref{eq2.5}) follows easily from (\ref{SepLemEqn}) by making the identification $k=r_n$ and using (\ref{ExpExc}) and (\ref{lemOBeq1}).  Additionally, if $\{X_n\}_{n=1}^{\infty}$ satisfies AIM($x_n$) then we can always find a sequence $p_n$ such that (\ref{lemOBeq1}) holds, for example, by taking $p_n = \floor*{ \{ n\, \text{max}(q_n, n\alpha_n)\}^{1/2} }$.  Thus the only assumption in Lemma \ref{lemOB} beyond common margins is that $\{X_n\}_{n=1}^{\infty}$ satisfies AIM($x_n$) for a sequence $x_n$ satisfying (\ref{ExpExc}). 

We can now state our main theorem.  

\begin{theorem} \label{MainThm}
\normalfont  Under the same assumptions as in Lemma \ref{lemOB}, we have
\begin{equation}
  \mathbb{P}(M_n \leq x_n) - \textnormal{exp}\bigg\{-\sum_{j=1}^{n} \mathbb{P}(X_j > x_n, M_{j,j+p_n} \leq x_n) \bigg\} \rightarrow 0,  \label{result1}
\end{equation}
and consequently
\begin{equation}
\mathbb{P}(M_n \leq x_n) - F(x_n)^{n\gamma_n}  \rightarrow 0,  \label{result2}
\end{equation}
where
\begin{equation}
  \gamma_n =  \frac{1}{n}\sum_{j=1}^{n} \mathbb{P}(M_{j,j+p_n} \leq x_n \mid X_j > x_n).   \label{gamma_n}
\end{equation}
\end{theorem}
As it was noted in Section \ref{Intro}, for independent sequences (\ref{ExpExc}) implies that $\mathbb{P}(M_n \leq x_n) \to e^{-\tau}$. 
For a random sequence satisfying the conditions of Lemma \ref{lemOB}, the following result gives a necessary and sufficient condition for the convergence of $\mathbb{P}(M_n \leq x_n).$
  
\begin{corollary}  \label{EasyCor}
\normalfont Under the same assumptions as in Lemma \ref{lemOB}, $\mathbb{P}(M_n \leq x_n)$ converges if and only if $ \lim_{n\to\infty} \gamma_n$ exists, where $\gamma_n$ is as in (\ref{gamma_n}), in which case
$\mathbb{P}(M_n \leq x_n) \to e^{-\tau\gamma}$ with $\gamma = \lim_{n\to\infty} \gamma_n\in [0,1].$
\end{corollary}
Corollary \ref{EasyCor} follows from (\ref{result2}) since $n\bar{F}(x_n) \to \tau$ if and only $F^n(x_n) \to e^{-\tau}$ which is easily seen by taking logs in the latter expression and using log($1-t) = -t + o(t)$ as $t\to 0$.

A basic question regarding the constant $\gamma$ appearing in Corollary \ref{EasyCor} is whether it is independent of the particular value of $\tau$ 
in (\ref{ExpExc}), i.e., do we obtain the same limiting value of $\gamma_n$ regardless of the specific sequence $x_n$ and $\tau$ used in (\ref{ExpExc})?  We will
see in Section \ref{PeriodicSec} that for sequences with periodic dependence this is indeed the case, and Theorem \ref{UniformConvThm}
gives sufficient conditions for this to hold more generally.

We now turn our attention to the conditional probabilities appearing in the summation (\ref{gamma_n}), which contain local information regarding the strength of extremal clustering in the sequence $\{X_n\}_{n=1}^{\infty}$. 

\begin{Definition}  \label{ClusteringFunctionDefn}
\normalfont Under the same assumptions as in Lemma \ref{lemOB},
let $\{f_n\}_{n=1}^{\infty}$ be the sequence of functions defined on $\mathbb{N}$ by 
\begin{equation}
f_n(i) = \theta_{i,n} = \mathbb{P}(M_{i,i+p_n} \leq x_n \mid X_i > x_n), \quad i\in \mathbb{N}.  \label{PartialClusteringFunction}
\end{equation}
We define the extremal clustering function of $\{X_n\}_{n=1}^{\infty}$ to be the function $\theta: \mathbb{N} \rightarrow [0, 1]$ given by
\begin{equation}
\theta_i = \lim_{n\to\infty} f_n(i)  \label{ClusteringFunction}
\end{equation}
provided the limit exists.
\end{Definition}
In the special case that the sequence $\{X_n\}_{n=1}^{\infty}$ is stationary, the extremal clustering function is simply a constant function equal to the extremal index of the sequence. In the general case, if we think of the index $i$ in $X_i$ as denoting time, then we may regard $\theta_i$ as the extremal index at time $i$.
The definition of $\theta_i$ entails pointwise convergence of the sequence of approximations $\{f_n\}_{n=1}^{\infty}$ in equation (\ref{PartialClusteringFunction}).  When there is a uniformity in this convergence and the extremal clustering function is Ces\`{a}ro summable we obtain the following result.
\begin{theorem}   \label{UniformConvThm} 
\normalfont Suppose $\{X_n\}_{n=1}^{\infty}$ satisfies AIM$(x_n)$ with $n\bar{F}(x_n) \to \tau > 0$. Assume that $\{\theta_i\}_{i=1}^{\infty}$ is Ces\`{a}ro summable and
\begin{equation}
\text{max}_{1\leq i \leq n}| \theta_i - \theta_{i,n} | \to 0 \quad \text{as\,} n\to \infty,  \label{PartialUnifConv}
\end{equation}
where $\theta_{i,n} = f_n(i)$ is as in equation (\ref{PartialClusteringFunction}). 
Then $\mathbb{P}(M_n \leq x_n) \to e^{-\tau\gamma}$ where 
\begin{equation}
\gamma = \lim_{n\to\infty} \frac{1}{n} \sum_{i=1}^{n} \theta_i. \label{CesaroMean}
\end{equation}
Moreover, if $\{y_n\}_{n=1}^{\infty}$ is a sequence of real numbers such that $n\bar{F}(y_n) \to \tau'$ with $\tau' \leq \tau$ then $\mathbb{P}(M_n \leq y_n) \to e^{-\tau'\gamma}$ 
with $\gamma$ as in (\ref{CesaroMean}). 
\end{theorem}

As with the constant $\gamma$ in Corollary \ref{EasyCor}, we may inquire as to whether the extremal clustering function is independent of the value of $\tau$ and sequence $x_n$ used in 
(\ref{ExpExc}).
 Although we do not attempt to answer this in full generality, we note that, as with the conditional probability formulation of the extremal index, 
for most sequences that are of practical interest, the formula defining $\theta_i$ may be reduced to a form that makes no explicit reference to the sequences 
$x_n$ and $p_n$.  For example, under the additional assumption due to \cite{Smith92} which requires that for any $x_n$ in Theorem \ref{MainThm} we have
\begin{equation}
\lim_{p\to\infty}\lim_{n\to\infty} \sum_{k=p}^{p_n}\mathbb{P}(X_{i+k} > x_n \mid X_i > x_n) = 0,   \label{SmithsCondition} 
\end{equation}
for each $i$, then (\ref{ClusteringFunction}) reduces to
\begin{equation}
\theta_i = \lim_{p\to\infty}\lim_{x\to x_F} \mathbb{P}(M_{i,i+p} \leq x \mid X_i > x).
\end{equation}
Another common assumption for statistical applications is the $D^{(k)}(x_n)$ condition 
of \cite{chern91} which we define below in a slightly modified form for our non-stationary setting. 
\begin{Definition} \label{D_k Defn} 
\normalfont
A sequence $\{X_n\}_{n=1}^{\infty}$ as in Theorem \ref{MainThm} is said to satisfy the $D^{(k)}(x_n)$ condition, where $k\in\mathbb{N}$, if
\begin{equation}
n\,\mathbb{P}(X_i > x_n, M_{i,i+k-1} \leq x_n, M_{i+k-1,i+p_n} > x_n) \to 0 \quad \text{as } n\to\infty  \label{D_k} 
\end{equation}
for each $i\in\mathbb{N}$.  For the case $k=1$, we define $M_{i,i} = -\infty$. 
\end{Definition}
Note that it is assumed in Definition \ref{D_k Defn} that $\{X_n\}_{n=1}^{\infty}$ satisfies AIM$(x_n)$ in conjunction with (\ref{D_k}).
Whereas (\ref{AIMalpha}) limits the degree of long range dependence in the sequence, (\ref{D_k}) is a local mixing condition that ensures that the probability of again exceeding the threshold $x_n$ in a block of $p_n$ observations, after dropping below it for $k-1$ consecutive observations 
falls to zero sufficiently rapidly as $n\to\infty$.  The case where $k=1$ implies that in the limit, any exceedances of a high threshold occur in isolation and is implied in the stationary case by the $D'(x_n)$ condition of \cite{leadling83}, Chapter 3.  One might expect that a more natural condition in our non-stationary setting would be to replace the constant 
$k$ in (\ref{D_k}) by $k_i$ to reflect possible variations in the strength of local dependence. However, when (\ref{D_k}) holds for some particular $k$, then it also holds for any other $k\textquotesingle$ with $k\textquotesingle > k,$ and so provided that the sequence $\{k_i\}_{i=1}^{\infty}$ is bounded we may set $k = \text{max}\{k_i : i\in\mathbb{N}\}$ and obtain (\ref{D_k}) for each $i$.  
Thus the assumption of a single value of $k$ in Definition \ref{D_k Defn} allows for variations in the strength of local dependence while at the same time restricting it to not persist too strongly to an arbitrary number of lags.  If whenever $x_n$ is a sequence as in Theorem \ref{MainThm} and the $D^{(k)}(x_n)$ condition holds then (\ref{ClusteringFunction}) reduces to
\begin{equation}
 \theta_i = \lim_{x\to x_F}\mathbb{P}(M_{i,i+k-1} \leq x \mid X_i > x).
\end{equation}

We will assume without further comment for the rest of the paper that the sequence $\{X_n\}_{n=1}^{\infty}$ has a well-defined extremal clustering function as may arise from assumptions (\ref{SmithsCondition}) or (\ref{D_k}).  

\subsection{Periodic dependence} \label{PeriodicSec}  

In this section we assume that the sequence $\{X_n\}_{n=1}^{\infty}$ has a more refined structure than in the previous sections,
namely that of periodic dependence, under which the results of Section \ref{Sec2.2} may be simplified considerably.  
\begin{Definition} \label{PeriodicDefn}
\normalfont A sequence $\{X_n\}_{n=1}^{\infty}$ with common marginal distributions is said to have periodic dependence if there exists $d\in \mathbb{N}$ such that 
$(X_{t_1},\ldots,X_{t_k})  \overset{D}{=}  (X_{t_1 + d},\ldots,X_{t_k + d})$ for all $t_1,\ldots,t_k\in \mathbb{N}.$  The smallest $d$ with this property is called the fundamental period. 
\end{Definition}
Whereas for a strictly stationary sequence an arbitrary shift in time leaves the finite-dimensional distributions unchanged, for a sequence with periodic dependence only time shifts that are a multiple of the fundamental period leave finite-dimensional distributions unchanged. 
 In particular, $M_{a,a+b} \overset{D}{=} M_{c,c+b}$ when $a\equiv c \Mod{d}.$ Such sequences often mimic the dependence structure of certain environmental time series where we might expect a fundamental period of one year.  

The following result concerning the convergence of $\mathbb{P}(M_n \leq x_n)$ shows that Theorem 3.7.1 of \cite{leadling83} for stationary sequences also holds for non-stationary sequences with periodic dependence.
\begin{theorem} \label{PeriodicThm}
\normalfont Let $\{X_n\}_{n=1}^{\infty}$ have periodic dependence and satisfy the conditions of Lemma \ref{lemOB}, with $x_n$ satisying (\ref{ExpExc})
for some $\tau > 0$. 
Suppose that $y_n = y_n(\tau')$ is a sequence of real numbers defined for each $\tau'$ with $0 < \tau' \leq \tau$ so that $n\bar{F}(y_n) \to \tau'$.
Then there exist constants $\gamma$ and $\gamma'$ with $0 \leq \gamma \leq \gamma' \leq 1$ such that 
\begin{align*}
 \text{lim\,sup}\,\,\mathbb{P}\{M_n \leq y_n(\tau')\} &= e^{-\tau'\gamma}  \\
 \text{lim\,inf}\,\,\mathbb{P}\{M_n \leq y_n(\tau')\}  &= e^{-\tau'\gamma'}
\end{align*}
for all  $0 < \tau' \leq \tau$. Hence if $\mathbb{P}\{M_n \leq y_n(\tau')\}$ converges for some $\tau'$ with $0 < \tau' \leq \tau$, then $\gamma = \gamma'$ and 
$\mathbb{P}\{M_n \leq y_n(\tau')\} \to e^{-\tau'\gamma}$ for all such $\tau'$. 
\end{theorem}
Although Theorem \ref{PeriodicThm} makes no reference to the extremal clustering function, when $\mathbb{P}(M_n \leq x_n)$ converges, the constant $\gamma$ in Theorem \ref{PeriodicThm} is identified by Corollary \ref{EasyCor} as $\gamma = \lim_{n\to\infty}\gamma_n$ with $\gamma_n$ as in equation (\ref{gamma_n}). Due to periodicity we obtain the simplified formula $\gamma = d^{-1}\sum_{i=1}^{d}\theta_i,$ and the extremal clustering function is determined by the $d$ values $\{\theta_i\}_{i=1}^{d}$ which repeat cyclically. Moreover, for sequences with periodic dependence, the convergence statement (\ref{PartialUnifConv}) can be strengthened to uniform convergence since $\text{sup}_{i\in\mathbb{N}}| \theta_i - \theta_{i,n} |  = 
\text{max}_{1\leq i \leq d}| \theta_i - \theta_{i,n} |. $ 

The following result is an immediate consequence of Theorem \ref{PeriodicThm}.  
\begin{corollary}
\normalfont Let $\{X_n\}_{n=1}^{\infty}$ have periodic dependence with common marginal distribution function $F$.  For each $\tau > 0$, let $x_n(\tau)$ be a sequence 
such that $n\bar{F}(x_n(\tau)) \to \tau$ and suppose that $\{X_n\}_{n=1}^{\infty}$ satisfies AIM$(x_n(\tau))$ for each such $\tau$.  If $\mathbb{P}\{M_n \leq x_n(\tau)\}$ converges for a single $\tau > 0$ then it converges for all $\tau > 0$, and in particular $\mathbb{P}\{M_n \leq x_n(\tau)\} \to e^{-\tau\gamma}$ for some $\gamma\in[0,1]$.
\end{corollary}

\subsection{Interexceedance times}   \label{Sec2.3}
  
\cite{ferrseg03} provided a method for estimating the extremal index of a stationary sequence without the need for identifying independent clusters of extremes.  This was achieved by considering the distribution of the time between two exceedances of a threshold $u$, i.e., 
\begin{equation}
T(u) = \text{min}\{ n\geq 1 : X_{n+1} > u \mid X_1 > u \},  \label{eq2.10}
\end{equation}
as $u$ approaches $x_F$.\  In particular, it was shown that the normalized interexceedance time
 $\bar{F}(x_n)T(x_n)$ converges in distribution as $n\to\infty$ to a mixture of a point mass at zero, with probability $1 - \theta$, and an exponential random variable with mean $\theta^{-1}$, with probability $\theta$. The mixture arises from the fact that the interexceedance times can be classified in to two categories: within cluster and between cluster times.  The mass at zero stems from the fact that the within cluster times, which tend to be small relative to the between cluster times, are dominated by the factor $\bar{F}(x_n)$. 

In the stationary case, conditioning on the event $X_1 > u$ in equation (\ref{eq2.10}) may be replaced with $X_i > u,$ and $X_{n+1}$ replaced by $X_{n+i},$ for any $i\in \mathbb{N},$  without affecting the distribution of $T(u)$.  In the non-stationary case we consider for each $i \in \mathbb{N}$ and threshold $u$, the random variable $T_i(u)$ defined by
\begin{equation}
T_i(u) = \text{min}\{n\geq 1   : X_{n+i} > u \mid X_i > u \}, \label{T_i Defn}
\end{equation}
whose distribution in general depends on $i$. We find that the distribution of $\bar{F}(x_n)T_i(x_n)$ converges as $n\to \infty$ to a mixture of a mass at zero, with probability $1-\theta_i$, and an exponential random variable with mean $\gamma^{-1}$, with probability $\theta_i$.  As in \cite{ferrseg03}, a slightly stronger mixing condition is required to derive this result than was needed for Theorem \ref{MainThm}. We denote
by $\mathcal{F}_{j_1,j_2}(u)$,  the $\sigma$-algebra generated by the events  $\{ X_i > u : j_1 \leq i \leq j_2 \}$, $j_1,j_2\in\mathbb{N}$,   
and we define the mixing coefficients
\begin{align}
\alpha^{*}_{n,q}(u) = \text{max}_{1\leq l\leq n-q}\, \text{sup} \mid \mathbb{P}(E_2 \mid E_1) - \mathbb{P}(E_2) \mid,   \label{SAIM}
\end{align} 
where the supremum is over all $E_1 \in \mathcal{F}_{1,l}(u)$ with $\mathbb{P}(E_1)>0$ and $E_2\in \mathcal{F}_{l+q,n}(u).$  
We will assume the existence of a sequence $q_n = o(n)$ such that $\alpha^{*}_{cn,q_n}(x_n) \to 0$ for all $c>0$. This implies that 
$\{X_n\}_{n=1}^{\infty}$ satisfies AIM($x_n$) with the same 
choice of $q_n$ and so we may find a sequence $p_n$ so that (\ref{lemOBeq1}) is satisfied.  We define $\theta_{i,n}$ to be as in equation (\ref{PartialClusteringFunction}) and assume a slightly 
stronger form of convergence than in (\ref{PartialUnifConv}) but weaker than uniform convergence $\text{sup}_{i\in\mathbb{N}}| \theta_i -\theta_{i,n}| \to 0.$ 
 
The limiting distribution of the normalized interexceedance times is given in Theorem \ref{IntExcThm}. 

\begin{theorem}  \label{IntExcThm} 
\normalfont Let $\{X_n\}_{n=1}^{\infty}$ be a sequence of random variables with common marginal distribution $F$ and $\{x_n\}_{n=1}^{\infty}$ a sequence of real numbers such that $n\bar{F}(x_n) \to \tau > 0$. Suppose that there is a sequence of positive integers $q_n = o(n)$  such that $\alpha^{*}_{cn,q_n}(x_n) \to 0$ and $\text{max}_{1\leq i \leq cn}| \theta_i -\theta_{i,n}| \to 0$ for all $c>0.$ 
Then, if $\{\theta_i\}_{i=1}^{\infty}$ is Ces\`{a}ro summable  we have, for each fixed $i\in\mathbb{N}$ and $t>0$
\begin{equation}
\mathbb{P}(\bar{F}(x_n)T_i(x_n) > t) \to  \theta_i \, \textnormal{exp}(-\gamma t). \label{eq2.12}
\end{equation}
\end{theorem}

\section{Estimation with a focus on periodic sequences}   \label{Sec3}

In this section we consider moment and maximum likelihood estimators for $\theta_i$ and $\gamma$ based on the limiting distribution of normalized interexceedance times given in Theorem \ref{IntExcThm}. 
We first show that the intervals estimator of \cite{ferrseg03} may be used to estimate $\theta_i$ and then consider likelihood based estimation along the lines of \cite{suv07}. 
For simplicity, we focus our discussion on the case of periodic dependence as in Definition \ref{PeriodicDefn}. Such an assumption reduces estimation of the extremal clustering function to estimating the vector $\bm{\theta} = (\theta_1,\ldots,\theta_d)$ with $\gamma = d^{-1}\sum_{i=1}^{d}\theta_i$ 
where $d$ is the fundamental period which, for simplicity, we assume to be known a-priori. Such an assumption is important for the moment based estimators of Section \ref{Sec3.2} where one needs replications of interexceedance times in order to use the estimators, but can easily be relaxed for likelihood based inference.  

\subsection{Moment based estimators}    \label{Sec3.2}

Theorem \ref{IntExcThm} implies that the first two moments of $\bar{F}(u)T_i(u)$ satisfy $\mathbb{E}\{\bar{F}(u)T_i(u) \} = \theta_i / \gamma + o(1) $ and $\mathbb{E}[\{ \bar{F}(u)T_i(u)  \}^2 ]  = 2\theta_i / \gamma^2 + o(1)$ as $u\to x_F$. Assuming the threshold is chosen to be suitably large so that the $o(1)$ terms can be 
neglected, these two equations can be solved with respect to the unknown parameters to give
\begin{align}
\gamma  = \frac{2\, \mathbb{E}(\bar{F}(u)T_i(u))}{\mathbb{E}(\{ \bar{F}(u)T_i(u)  \}^2)}\,\,\, \text{and } \,
\theta_i  = \frac{ 2 \{\mathbb{E}(\bar{F}(u)T_i(u)) \}^2   }{\mathbb{E}(\{ \bar{F}(u)T_i(u)  \}^2) }  =   \frac{ 2 \{\mathbb{E}(T_i(u)) \}^2   }{\mathbb{E}(\{T_i(u) \}^2)}\mathbin{\raisebox{0.5ex}{,}}   \quad 1\leq i \leq d.    \label{eq3.1}
\end{align}
A complication that arises in the non-stationary setting is that, since $\theta_i$ is defined via a conditional probability 
given the event $X_i > u,$ if $X_i$ does not exceed the threshold $u$ then there are no interexceedance times to estimate $\theta_i$.  This problem doesn't arise in the stationary case where every interexceedance time may be used to estimate the extremal index $\theta$. 

In order to estimate $\theta_i$ then, it is natural to assume that the extremal clustering function is structured in some way, e.g., periodic or piecewise constant. Making such an assumption allows us to use multiple interexceedance times to estimate $\theta_i$. Focusing on the case where $\{X_n\}_{n=1}^{\infty}$ has periodic dependence with fundamental period $d$, all exceedances of the threshold $u$ occuring at points that are separated by a multiple of $d$ give rise to interexceedance times that may be used to estimate the same value of the extremal clustering function.  More precisely, suppose that $X_1,\ldots,X_n$ is a sample of size $n$ of the process with exceedance times $E = \{1\leq i\leq n : X_i > u\},$ and corresponding interexceedance times 
$I = \{T_i(u): i\in E\, \backslash \{\text{max}(E)\}  \},$ with $T_i(u)$ as in equation (\ref{T_i Defn}).
The set of interexceedance times that may be used for estimating $\theta_i$ is the subset $I_i \subseteq I$ defined by 
$I_i = \{T_j(u) \in I : j\equiv i \Mod{d} \}.$  If $| I_i\,| = N_i$, then we may relabel the elements of $I_i$ as 
$I_i = \{T_i^{(j)} \}_{j=1}^{N_i}$ where now the subscript remains fixed.  Making further, more refined assumptions regarding the nature of the periodicity of the process under consideration may give rise to different sets $I_i$.
For example, in an environmental time series setting it may be reasonable to assume that the extremal clustering function is piecewise constant within months or seasons, so that all interexceedance times that correspond to exceedances within the same calendar month or season belong to the same $I_i$.

Equation $(\ref{eq3.1})$ suggests the estimator 
\begin{equation}
\hat{\theta}_i = \frac{2 \big(\sum_{j=1}^{N_i} T_i^{(j)} \big)^2   }{ N_i \sum_{j=1}^{N_i} (T_i^{(j)} )^2  }\mathbin{\raisebox{0.5ex}{,}}   \label{IntHat}
\end{equation}
whose bias we now investigate.  From ($\ref{eq2.12}$) we have that for $n \in \mathbb{N}$ 
\begin{equation*}
\mathbb{P}( T_i(x_n) > n )  = \theta_i \, F(x_n)^{n\gamma} + o(1), \\
\end{equation*}
which motivates consideration of the positive integer valued random variable $T$ defined by  
\begin{equation*}
\mathbb{P}(T> n) = \theta_i \, p^{n\gamma}, \quad \text{for } n\geq 1,   
\end{equation*}
where $p\in (0,1)$ and $\theta_i, \gamma \in (0,1]$ and we may identify $p$ with $F(x_n).$   In a similar manner to \cite{ferrseg03}, we find that 
$\mathbb{E}(T) = 1 + \theta_i p^{\gamma}  (1 - p^{\gamma} )^{-1}  $ and 
$\mathbb{E}(T^2) = 1 + \theta_i p^{\gamma}  (1 - p^{\gamma} )^{-1} + 2\, \theta_i \, p^{\gamma} (1 - p^{\gamma})^{-2}$,
so that upon simplification we find that
\begin{equation}
\frac{ 2\big\{\mathbb{E}(T) \big\}^2 }{\mathbb{E}(T^2) } = \frac{2( 1 - p^{\gamma} + \theta_i p^{\gamma} )^2  }{ (1-p^{\gamma})^2 + \theta_i p^{\gamma}(1-p^{\gamma} ) + 2\, \theta_i p^{\gamma} }\cdot  \label{eq3.6}
\end{equation}
A Taylor expansion of the right hand side of equation (\ref{eq3.6}) around $p=1$ gives 
\begin{equation*}
\frac{ 2\big\{\mathbb{E}(T) \big\}^2 }{\mathbb{E}(T^2) } =  \theta_i +\gamma(3/2 - 2\theta_i)(1 - p) + O\big\{ (1-p )^2 \big\}, \quad \text{as } p \to 1,
\end{equation*}
so that the first order bias of $\hat{\theta}_i$ is $\gamma(3/2 - 2\theta_i)\bar{F}(x_n)$.
On the other hand, since
\begin{equation*}
\theta_i =  \frac{2 \big\{\mathbb{E}(T-1) \big\}^2}{\mathbb{E}\{(T-1)(T-2)\}}\,\mathbin{\raisebox{0.5ex}{,}}
\end{equation*}
this motivates the estimator 
\begin{equation}
\tilde{\theta}_i =  \frac{2 \sum_{j=1}^{N_i} (T_i^{(j)} - 1)^2  } { N_i \sum_{j=1}^{N_i} (T_i^{(j)} -1)(T_i^{(j)} - 2)}\,\mathbin{\raisebox{0.5ex}{,}}   \label{IntTilde}
\end{equation}
whose first order bias is zero.\ This estimator forms the key component of the intervals estimator of \cite{ferrseg03}, which we can use to estimate $\theta_i$.  We note that $\tilde{\theta}_i $ may take values greater than 1 and is not defined if max$(I_i) \leq 2$ as then the denominator in (\ref{IntTilde}) is zero.  In order to deal with these cases, the intervals estimator $\theta_i^*$ of $\theta_i$ is defined as
\begin{displaymath}
\theta_i^* = 
\begin{cases}
\text{min}\{1, \hat{\theta}_i\} \quad \text{if}\,\, \text{max}(I_i) \leq 2, \\
\text{min}\{1, \tilde{\theta}_i \}  \quad \text{if}\,\, \text{max}(I_i) > 2.
\end{cases}
\end{displaymath}
 
While equation (\ref{eq3.1}) also suggests an estimator for $\gamma$, this is based only on the interexceedances relevant to estimating $\theta_i$ and also requires an estimate of $\bar{F}(u)$. One possibility is to obtain $d$ such estimates and take the mean of these as the estimate of $\gamma$.  However, this estimator need not respect the relation $\gamma = d^{-1}\sum_{i=1}^{d} \theta_i$, a consequence of the fact that we dropped the $o(1)$ terms when solving the first two moment equations.  In the examples that we consider in Section \ref{Sec4}, we estimate $\gamma$ using the mean of the estimates for the $\theta_i$ values.    

\subsection{Maximum likelihood estimation}    \label{Sec3.3}

Theorem \ref{IntExcThm} also allows for the construction of the likelihood function for the vector of unknown parameters.  This is an attractive approach due to the modelling possibilities that become available, however, as discussed in \cite{ferrseg03} in the stationary case, problems arise with maximum likelihood estimation due to uncertainty in how to assign interexceedance times to the components of the limiting mixture distribution.  Since the asymptotically valid likelihood is used as an approximation at some subasymptotic threshold $u$, all observed normalized interexceedance times are strictly positive.  Assigning all interexceedance times to the exponential part of the limiting mixture means that they are all being classified as between cluster times.\  This is tantamount to exceedances of a large threshold occuring in isolation, and so the maximum likelihood estimator based on this, typically misspecified, likelihood converges in probability to 1 regardless of the true underlying value of $\theta$. 

This problem was addressed in \cite{suv07} for sequences satisfying the $D^{(2)}(x_n)$ condition, i.e., the case $k=2$ in (\ref{D_k}).  For such sequences,
 in the limit as $n\to\infty$, exceedances above $x_n$ cluster into independent groups of consecutive exceedances, so that all observed interexceedance times equal to one are assigned to the zero component of the mixture likelihood.  On the other hand, all interexceedance times greater than one are assigned to the exponential component of the likelihood. It was found that, when the  $D^{(2)}(x_n)$ condition is satisfied,
maximum likelihood estimation outperforms the intervals estimator in terms of lower root mean squared error.
  The consecutive exceedances model of clusters implied by $D^{(2)}(x_n)$ is in contrast to the general situation where within clusters, exceedances may be separated by observations that fall below the threshold.

If we were to make the $D^{(2)}(x_n)$ assumption in our non-stationary setting, so that the 
consecutive exceedances model for clusters is accurate, then with $I_i = \{T_i^{(j)} \}_{j=1}^{N_i}$ the interexceedance times relevant for estimating $\theta_i$ as in Section \ref{Sec3.2},
we obtain the likelihood function as
\begin{equation*}
L(\bm{\theta} ; I ) = \prod_{i=1}^{d} L_i(\bm{\theta} ; I_i ) 
\end{equation*}
where $I = \cup_{i=1}^{d}I_i$ is the set of all interexceedance times and
\begin{equation*}
 L_i(\bm{\theta} ; I_i)  = \prod_{j=1}^{N_i} (1-\theta_i)^{\mathbbm{1}[T^{(j)}_i=1] }\big\{\theta_i \gamma \text{exp}(-\gamma \bar{F}(x_n)T_i^{(j)} )  \big\}^{\mathbbm{1}[T^{(j)}_i >1]}.
\end{equation*}
The full log-likelihood is then
\begin{align}
l(\bm{\theta} ; I) = & \sum_{i=1}^{d}(N_i -  n_i)\text{log}(1-\theta_i) + \sum_{i=1}^{d}n_i\, \text{log}(\theta_i) + \big(\sum_{i=1}^{d}n_i  \big)\text{log}(\gamma)  \notag \\
& -\gamma \bar{F}(x_n)\sum_{i=1}^{d}\sum_{j=1}^{N_i}(T_i^{(j)}-1) - \gamma\bar{F}(x_n)\sum_{i=1}^{d} n_i, \label{loglik}
\end{align}
where $\gamma = d^{-1} \sum_{i=1}^{d}\theta_i,$ $n_i = \sum_{j=1}^{N_i}\mathbbm{1}[T^{(j)}_i > 1],$ and in practice $\bar{F}(x_n)$ must be replaced with an estimate.
Unlike in the stationary case, the likelihood equations don't have a closed form solution, essentially due to the dependence of 
$\gamma$ on all the $\theta_i$. Equation (\ref{loglik}), however, is easily optimized numerically provided $d$ is not too large.  If $d$ is large, it is more natural to parameterize $\theta_i$ in terms of a small number of parameters which we may estimate by maximum likelihood or consider non-parametric estimation along the lines of \cite{einmetal16}. 

We may generalise this idea and assign all interexceedance times less than or equal to some value $k$ to the zero component of the likelihood, so that the corresponding expression for $L_i$ becomes
\begin{equation}
 L_i(\bm{\theta} ; I_i )  = \prod_{j=1}^{N_i} (1-\theta_i)^{\mathbbm{1}[T^{(j)}_i \leq k] }\big\{\theta_i \gamma \text{exp}(-\gamma \bar{F}(x_n)T_i^{(j)} )\big\}^{\mathbbm{1}[T^{(j)}_i>k] }. \label{L_i}
\end{equation}
 This may be justified by the assumption that the sequence satisfies the  $D^{(k+1)}(x_n)$ condition.
Selection of an appropriate value of $k$ is equivalent to the selection of the run length for the runs estimator, and this problem is considered 
in the stationary case in \cite{suv10} and \cite{cai19}. 
However, in a non-stationary setting, where the clustering characteristics of the sequence may change in time, the appropriate value of $k$ may also be time varying, so that $k$ may be replaced with $k_i$ in equation (\ref{L_i}).  
Although, as was discussed  in Section \ref{Sec3.2}, we may take a constant value of $k$ in the definition of $D^{(k)}(x_n),$ for the purposes
of estimating $\theta_i$, one wants to select for each $i$, the smallest $k=k_i$ such that (\ref{D_k}) is satisfied \citep{Hsing93}.
 If too small a value is selected for $k_i$ then some of the interexceedance times may be wrongly assigned to the exponential component of the likelihood leading to an overestimate of
 $\theta_i$ whereas if $k_i$ is selected to be too large then we tend to underestimate $\theta_i$.  

\section{Examples} \label{Sec4}

In this section we consider two simple examples of non-stationary Markov sequences with a periodic dependence structure and common marginal distributions. 
The first example we consider is the Gaussian autoregressive model  
\begin{equation}
X_{n+1} = \rho_n X_n+ \epsilon_n, \quad  n \geq 1,  \label{GaussianMod}
\end{equation}
where $\epsilon_n \sim N(0, 1-\rho_n^2)$,  $|\rho_n | < 1$ and $X_1 \sim N(0,1)$. 
In our second example, we consider a model where $(X_n, X_{n+1})$ follow a bivariate logistic distribution with joint distribution function
\begin{equation} 
F(x_n, x_{n+1}) = \text{exp}\big\{- ( x_n^{-1/\alpha_n} + x_{n+1}^{-1/\alpha_n})^{\alpha_n} \big\}, \quad n\geq 1 , \label{LogisticMod} 
\end{equation}
$\alpha_n \in (0,1]$ and $X_1 \sim \text{Fr\'{e}chet}(1)$ so that $\mathbb{P}(X_1 \leq x) = e^{-1/x}, x \geq 0.$
For the Gaussian model, no limiting extremal clustering occurs at any point in the sequence, so that $\theta_i = 1$ for each $i$, in contrast to the logistic model where $\theta_i < 1$ for each $i$.

For sufficiently well behaved stationary Markov sequences, mixing conditions much stronger than those considered in Section \ref{Sec2} hold.  
For example, for the stationary Gaussian autoregressive sequence, with $\rho_n = \rho$ in (\ref{GaussianMod}) for all $n\geq 1,$ Theorems 1 and 2 from \cite{Athreya86} give that the mixing conditions of Theorem \ref{MainThm} and Theorem \ref{IntExcThm}  hold for any sequence $q_n$ such that $q_n \to \infty$, $q_n = o(n)$, for any 
$x_n$. Analogous results also hold for the non-stationary models that we consider in this section, see for example \cite{Bradley05} Theorem 3.3 and \cite{davydov73} Theorem 4.     
 
\subsection{Gaussian autoregressive model}  \label{Sec4.2}

Stationary sequences $\{X_n\}_{n=1}^{\infty}$ where each $X_i$ is a standard Gaussian random variable, are extensively studied in Chapter 4 of \cite{leadling83}.  It is shown there that if the lag $n$ autocorrelation $\rho(n)$ satisfies $\rho(n)\,\text{log\,}n \to 0$, then the extremal index $\theta$ of the sequence equals one and so no limiting extremal clustering occurs.  Thus, the stationary autoregressive sequence with $\rho_n = \rho$ in (\ref{GaussianMod}) for all $n\geq 1$ has extremal index one, provided $\rho < 1.$  This is a special case of the more general result that a stationary asymptotically independent Markov sequence has an extremal index of one \citep{Smith92}. We say that the stationary sequence $\{X_n\}_{n=1}^{\infty}$ is asymptotically independent at lag $k$ if \raisebox{2pt}{$\chi$}$(k) = 0$ where
\begin{displaymath}
\raisebox{2pt}{$\chi$}(k) = \lim_{u\to x_F}\mathbb{P}(X_{n+k}  > u \mid X_n > u), \quad k\geq 1,
\end{displaymath}
and asymptotically independent if \raisebox{2pt}{$\chi$}$(k) = 0$ for all $k$ \citep{ledtawn03}.

Here, we consider the non-stationary autoregressive model (\ref{GaussianMod}) and specify a periodic lag one correlation function  $ \rho_{n+1} = 0.5 + 0.25\,\text{sin}(2\pi n/7)$ for $n \geq 0$.  
Applying Theorem 6.3.4 of \cite{leadling83}, and comparing the non-stationary sequence to an independent standard Gaussian sequence, we deduce that $\mathbb{P}(M_n \leq x_n) - \Phi(x_n)^n \to 0$ as $n \to \infty$  where $\Phi$ is the standard Gaussian distribution function, and thus conclude that $\gamma = 1$ and $\theta_i = 1$ for $i=1,\ldots,7.$  
The same conclusion may also be drawn by applying Theorem 4.1 of \cite{husl83}, which shows that if $x_n$ satisfies (\ref{ExpExc}) then $\mathbb{P}(M_n \leq x_n) \to e^{-\tau}.$ 

We simulated 1000 realizations of this sequence of length $10^4$ and, for each realization, estimated $\theta_1,\ldots, \theta_7$ and $\gamma$ for a range of high thresholds, using both the intervals estimator and maximum likelihood with $k$ in equation (\ref{L_i}) equal to zero and one. We then repeated this procedure for sequences of length $10^5$ and $10^6$.  We found that the maximum likelihood estimator with $k=0$ gave by far the best performance as measured by root mean squared error in $\gamma$. In fact, in this case the 0.025 and 0.975 empirical quantiles of the estimated values of $\gamma$ were both 1 to two decimal places in all simulations.  This is not surprising since selecting $k=0$ ensures that all interexceedance times have the correct asymptotic classification as between cluster times. However, in a real data example such a level of prior knowledge regarding asymptotic independence is not realistic and would render estimation redundant. 
Although maximum likelihood estimation with $k=1$ performed slightly poorer than the intervals estimator, both methods produced broadly similar results. 

Table \ref{T1} shows the 0.025 and 0.975 empirical quantiles of the parameter estimates obtained using the intervals estimator. In the table, $u=q_p$ corresponds to the threshold that there is probability $p$ of exceeding at each time point i.e., $\mathbb{P}(X_i > q_p) = p$.  Although the true value of each $\theta_i$ is 1, so that no extremal clustering occurs in the limit
as $u\to \infty$, clustering may occur at subasymptotic levels.  Moreover, there will tend to be more subasymptotic clustering in the sequence at points with a larger lag one autocorrelation, i.e., larger $\rho_i$.  This point has been thoroughly discussed in the context of stationary sequences and estimation of the extremal index \citep{NavTawn, eastoetawn12} and leads to the notion of a subasymptotic or threshold based extremal index.  

\begin{table}  
\centering
\caption{0.025 and 0.975 empirical quantiles of the estimates of $\theta_1,\ldots,\theta_7, \gamma$, in the Gaussian autoregressive model using the intervals estimator. These are based on 1000 realizations of the process for different sample sizes $n$ and thresholds $u$. } 
 \label{T1}
\begin{tabular}{|c | c | c | c | c | c | c | c | c | c | }
 \hline 
  $n$  & $u$  & $\theta_1$ & $\theta_2$ & $\theta_3$ & $\theta_4$ & $\theta_5$ & $\theta_6$ & $\theta_7$  & $ \gamma $  \\
\hline
 $10^4$ & $q_{0.10}$ & $ .53, .84  $ & $ .39, .66 $ & $ .34, .62 $ &  $ .46, .77 $ & $.62, .93$ & $.69, 1.0$ & $.66, 1.0$ & $.62, .74 $  \\
\hline
$10^4$  & $q_{0.05}$ & $.53, .96$ & $.39, .78$ & $.35, .74$ & $.48, .91$ &  $.61, 1.0$  & $.70, 1.0$ & $ .64, 1.0$ & .66, .82  \\
\hline
$10^4$ &  $q_{0.01}$ & $.54, 1.0 $ & $.41, 1.0$ & $.39, 1.0$ & $.51, 1.0 $ &  $.61, 1.0$ & $.62, 1.0$ & $.60, 1.0$ & .75, .97  \\   
\hline
$10^5$ &  $q_{0.10}$ & $.62, .73 $ & $.46, .56$ & $ .42, .51$ & $.55, .65$ &  $.71, .81  $ & $.78, .90$ & $ .77, .87$ &  .65, .69 \\   
\hline
$10^5$ &   $q_{0.05}$ & $.66, .80  $ & $.50, .63$ & $.46, .59$ & $.60, .73$ &  $.75, .90$ & $.81, .97$ & $.79, .94$ &  .70, .75  \\   
\hline
$10^5$ & $q_{0.01}$ & $.69, 1.0 $ & $.55, .84$ & $.52, .81$ & $.65, .96$ &  $.76, 1.0$ & $.80, 1.0$ & $.78, 1.0$ &  .78, .88  \\  
\hline
$10^6$ &  $q_{0.05}$ & $.71, .75$ & $.54, .58$ & $.51, .55$ & $.64, .69$ & $.80, .85$ & $.87, .91$ & $.84, .89$ & .71, .73 \\
\hline
$10^6$ &  $q_{0.01}$ & $.78, .89$  &  $.63, .73$  &   $.60, .70$  &  $.74, .84$  &  $.87, .97$  &  $.90, 1.0 $ &  $.88, .99$  &.81, .84 \\
\hline
$10^6$ &  $q_{0.001}$ & $.77 , 1.0$  &  $.65, .98$  &   $.65, .94$  &  $.75, 1.0$  &  $.82, 1.0$  &  $.83, 1.0$ &  $.83, 1.0$  & .86, .95 \\ 
\hline
\end{tabular}
\end{table}

\subsection{Bivariate logistic dependence}   \label{Sec4.3}

The stationary logistic model, that is, (\ref{LogisticMod}) with $\alpha_n = \alpha$ for all $n\geq 1$, has been thoroughly studied 
\citep{Smith97, ledtawn03, suv07}.
The parameter $\alpha$ controls the strength of dependence between adjacent terms in the sequence, with $\alpha=1$ corresponding to independence and $\alpha \to 0$ giving complete dependence.   Such a sequence exhibits asymptotic dependence provided $\alpha < 1$, in particular, $\lim_{u\to \infty} \mathbb{P}(X_{n+1} > u \mid X_n > u)  = 2 - 2^{\alpha}.$  By exploiting the Markov structure of the sequence, precise calculation of $\theta$ can be achieved using the numerical methods described in \cite{Smith92}, where it is found for example that the sequence with $\alpha = 1/2$ has $\theta = 0.328,$ 
and moreover, equation (\ref{SmithsCondition}) is shown to hold for all $\alpha \in (0,1]$.
The case of $\alpha = 1/2$ is also considered in \cite{suv07} where, based on diagnostic plots, it is concluded that the $D^{(2)}(x_n)$ condition is not satisfied for this sequence, and moreover, the maximum likelihood estimator for $\theta$ based on a run length of $k=1$ has bias of around $20\%$. \cite{suv10} find that a more suitable run length is $k=5$, and in this case the maximum likelihood estimator for $\theta$ has lower root mean squared error than the intervals estimator.  Smaller values of $\alpha$ will tend to be associated with larger values of the run length $k$, though the precise nature of this relation is unclear. 

We consider the non-stationary logistic model (\ref{LogisticMod}) with $\alpha_{n+1} = 0.5 + 0.25\, \text{sin}(2\pi n /7) $ for $n\geq 0.$  Note that although we have specified the same parametric form for the dependence parameters $\alpha$ as in the previous example for $\rho$, the two parameters are not directly comparable. 
We simulated 1000 realizations of this process, of lengths $10^4$ and $10^5$, and estimated $\theta_1,\ldots,\theta_7$ using maximum likelihood
with $k=5,$ at a range of different thresholds. Table \ref{T3} shows, for the different sample sizes and thresholds considered,  the 0.025 and 0.975 
empirical quantiles of the parameter estimates obtained from this simulation.  Although the exact values of the parameters are unknown, making evaluation of any estimators performance impossible, an upper bound for $\theta_i$ is easily obtained as $ \lim_{u\to\infty} \mathbb{P}(X_{i+1} \leq u \mid X_i > u) = 2^{\alpha_i} - 1$. 
In our case this gives the bounds 
$(\theta_1,\ldots, \theta_7) \leq (0.41, 0.62, 0.67, 0.52, 0.31, 0.19, 0.24)$ and $\gamma \leq 0.42$ where the relation $\leq$ is interpreted componentwise.
It is conceivable that the methods in \cite{Smith92} could be adapted to the non-stationary case to allow exact computation of $\theta_i$ though we do not pursue this direction here.

We also considered estimation of $\theta_i$ using the intervals estimator and obtained similar results to the maximum likelihood estimates. The median value of the 1000 estimates for each parameter using the different methods of estimation are shown in Figure \ref{Fig1} for the sample size of $10^5$ and threshold $q_{0.05}$. The estimators clearly recover the periodicity in the dependence structure of the sequence and, on average, respect the upper bound for $\theta_i$ of $2^{\alpha_i} - 1$.

\begin{table}  
\centering
\caption{0.025 and 0.975 empirical quantiles of the estimates of  $\theta_1,\ldots,\theta_7, \gamma$, in the logistic time series model using maximum likelihood with $k=5$.  These are based on 1000 realizations of the process for different sample sizes $n$ and thresholds $u$.  }
\label{T3}
\begin{tabular}{|c | c | c | c | c | c | c | c | c | c | }
 \hline 
  $n$  & $u$  & $\theta_1$ & $\theta_2$ & $\theta_3$ & $\theta_4$ & $\theta_5$ & $\theta_6$ & $\theta_7$  & $ \gamma $   \\
\hline
 $10^4$ & $q_{0.10}$ &   $.19, .39$ & $.38, .58$ & $.44, .65$ & $.30, .52$ &  $.12, .29$ & $.05, .19$ & $ .08, .23$ & .28, .35  \\
\hline
$10^4$  & $q_{0.05}$ & $.21, 46 $ & $.41, .67$ & $.46, .72$ & $.30, .55$ &  $.11, .34$ & $.05, .22$ & $.08, .27$ &  .29, .39  \\
\hline
$10^4$ &  $q_{0.01}$ & $.11, .65$ & $.30, .85$ & $.37, .92$ & $.19, .74$ &  $.03,  .50$ & $.00, .34$ & $.00, .41$ & .27, .48   \\   
\hline
$10^5$ &  $q_{0.10}$ & $.22, .35 $ & $.42, 54$ & $.48, .62 $ & $.32, .49$ &  $.16, .24$ & $ .08, .17$ & $.11, .20$ & .29, .33  \\   
\hline
$10^5$ &   $q_{0.05}$ & $.27, .42 $ & $ .46, .61 $ & $.53, .66$ & $.36, .51$ &  $.17, .29$ & $.08, .19$ & $ .10, .21$ &  .32, .37   \\   
\hline
$10^5$ & $q_{0.01}$ & $.27, .46$ & $ .48, .67 $ & $.53, .73$ & $.35, .55$ &  $.15, .32$ & $.07, .20$ & $ .11, .26$ & .33, .40   \\  
\hline 
\end{tabular}
 \label{BVLtable}
\end{table}

\begin{figure}[h]
\begin{center}
\includegraphics[scale=0.75]{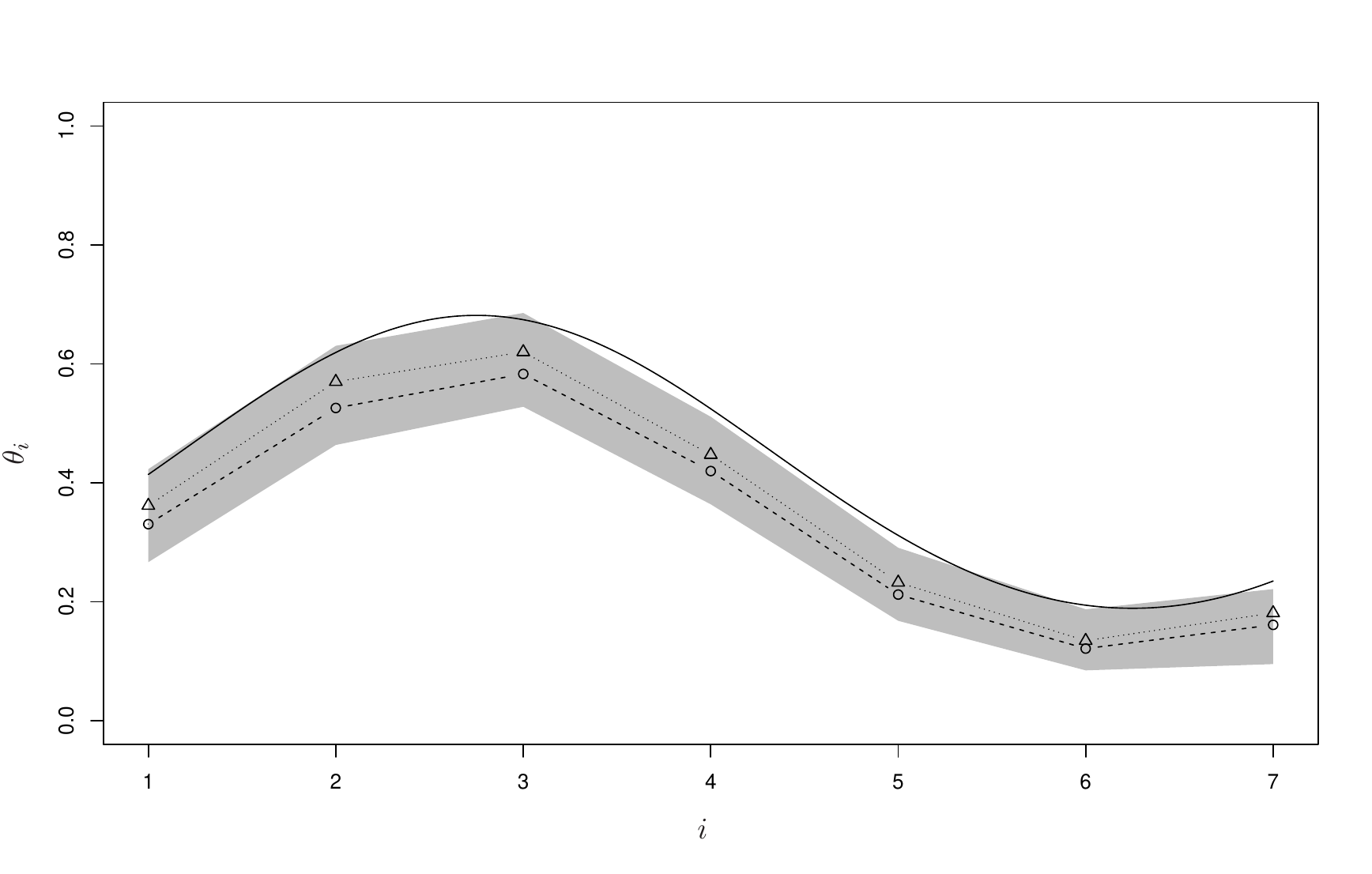}
\caption{Illustration of estimators (triangles: intervals estimator, and circles: maximum likelihood with $k=5$) obtained from $10^3$ 
realizations from the non-stationary logistic model of length $10^5$ and threshold $q_{0.05}$. 
The marked points correspond to the median estimate from $10^3$ realizations of the model. The grey region contains the 0.025 and 0.975 empirical quantiles 
of the parameter estimates using both the intervals and maximum likelihood estimators.  It is constructed by taking the pointwise maxima of the 0.975 quantiles (upper boundary) and pointwise minima of the 0.025 quantiles (lower boundary). 
The solid black curve shows the upper bound for $\theta_i$ of $2^{\alpha_i}-1$. }
\label{Fig1}
\end{center}
\end{figure}
	
\section{Proofs}   \label{SecProofs}

\subsection{Auxiliary results} 

In this section we state and prove some Lemmas that are required in the proof of Theorem \ref{MainThm}. 

\begin{lemma}  \label{L5.1}   \normalfont
Let $\{t_n\}_{n=1}^{\infty}$ be a sequence of positive integers and $a_{i,n}$, $i, n \in \mathbb{N},$ an array of non-negative real numbers
such that $t_n\to\infty$ and $A_n = \max_{1\leq i \leq t_n} a_{i,n} \to 0$ as $n\to\infty$. Then, 
\begin{align}
& \sum_{i=1}^{t_n} a_{i,n} \to \tau \geq 0 \label{eq05}   \\  
\intertext{if and only if} 
& \prod_{i=1}^{t_n} (1 - a_{i,n}) \to e^{-\tau}.  \label{eq06}
\end{align}
\end{lemma}
\begin{proof}
Using the fact that $\text{log}(1-x) = -x + R(x)$ where $|R(x)| < Cx^2$, for sufficiently small $x > 0$, for some $C > 0$, we have
$\sum_{i=1}^{t_n}\text{log}(1-a_{i,n}) = -\sum_{i=1}^{t_n}a_{i,n} + \sum_{i=1}^{t_n}R(a_{i,n}) $
and
\begin{equation}
 |\sum_{i=1}^{t_n}R(a_{i,n})| \leq C\sum_{i=1}^{t_n}a_{i,n}^2 \leq CA_n\sum_{i=1}^{t_n}a_{i,n},
\end{equation}
so that $\sum_{i=1}^{t_n} \text{log}(1-a_{i,n}) = -\big(\sum_{i=1}^{t_n}a_{i,n} \big)\big(1 + o(1)\big)$ or equivalently
\begin{equation}
 \text{log}\prod_{i=1}^{t_n}(1-a_{i,n}) = -\big(\sum_{i=1}^{t_n}a_{i,n} \big)\big(1 + o(1)\big),
\end{equation}
from which the result follows.
\end{proof}

\begin{lemma}  \label{L5.2}    \normalfont
Let $\{t_n\}_{n=1}^{\infty}$ be a sequence of positive integers and $a_{i,n}$, $i, n \in \mathbb{N},$ an array of non-negative real numbers
such that $t_n\to\infty$,  $A_n = \max_{1\leq i \leq t_n} a_{i,n} \to 0$ and $\sum_{i=1}^{t_n}a_{i,n}$ is bounded above as $n\to\infty$. Then, 
\begin{equation} 
\prod_{i=1}^{t_n} (1 - a_{i,n}) - \text{exp}\bigg({-\sum_{i=1}^{t_n} a_{i,n}} \bigg)  \to 0.
\end{equation} 
\end{lemma}

\begin{proof}
This follows from Lemma \ref{L5.1} by considering subsequences along which $\sum_{i=1}^{t_n} a_{i,n}$ converges. 
\end{proof}
\begin{lemma}  \label{L5.3}  \normalfont
Let $g:\mathbb{R} \to \mathbb{R}$ be a bounded function.  If $f(x) = A(x)g(x)$ and $\text{lim}_{x\to\infty}A(x) = 1,$ then $f(x) = g(x) + o(1)$ as $x\to \infty$.
\end{lemma}
\begin{proof}
As $g$ is bounded, $\exists M >0$ such that $|g(x)| < M,$ $\forall x \in \mathbb{R}$. 
Now let $\epsilon > 0.$  As  $\text{lim}_{x\to\infty}A(x) = 1,$ $\exists x_0$ such that
\begin{equation*}
|A(x) - 1 | < \epsilon/M \quad \text{for} \,\, x>x_0.
\end{equation*}
Then for $x>x_0$ 
\begin{equation*}
| f(x) - g(x) | = |g(x)||A(x) - 1 | < M\epsilon/M = \epsilon.
\end{equation*}
\end{proof}

\begin{lemma}  \label{L5.4}     \normalfont
Let $\{X_n\}_{n=1}^{\infty}$, $x_n$, $p_n$ and $q_n$ be as in Lemma \ref{lemOB} such that (\ref{ExpExc}) holds and assume 
$\mathbb{P}(M_n \leq x_n) \to L \in (0,1)$.  Let $s_n$ be such that $p_n = o(s_n), s_n = o(n)$ and $t_n = \floor*{n/(s_n + q_n)}$.  Then 
\begin{equation}
\sum_{i=1}^{t_n}  \mathbb{P}(M^i_{s_n-p_n,s_n} > x_n) = o\bigg(\sum_{i=1}^{t_n}\mathbb{P}(M^i_{0,s_n-p_n} > x_n, M^i_{s_n-p_n, s_n} \leq x_n )\bigg) \label{NewCond2}
\end{equation}
where  $M^i_{j,k} = \text{max}\,\{X^i_{j+1},\ldots, X^i_{k} \}$ and $X^i_j = X_{(i-1)\,(s_n+q_n)+j}$.
\end{lemma}

\begin{proof}
We first note that Lemma \ref{lemOB} also holds with blocks of length $s_n$, i.e., with $s_n$ in place of $p_n$ in the definition of $A_i$ in equation (\ref{A_i}) and $t_n$ in place of $r_n$.  Thus from equation (\ref{eq2.5}), with blocks of length $s_n$, we have that 
$ \mathbb{P}(M_n \leq x_n) - \prod_{i=1}^{t_n}\mathbb{P}(M(A_i) \leq x_n) \to 0$ so that $\prod_{i=1}^{t_n}\mathbb{P}(M(A_i) \leq x_n) \to L \in (0,1),$ or equivalently
\begin{equation}
 \sum_{i=1}^{t_n}\text{log}\big(1 - \mathbb{P}(M(A_i) > x_n)\big)\to \text{log}(L).  \label{SumForm}
\end{equation}
Now we note that $\text{max}_{1\leq i \leq t_n}\mathbb{P}(M(A_i) > x_n) \to 0$ since $\mathbb{P}(M(A_i) > x_n) \leq s_n\bar{F}(x_n)$ and $s_n=o(n)$ and $\bar{F}(x_n) = \tau\,n^{-1}+o(n^{-1})$.  Thus, using $\text{log}(1-t) = -t + o(t)$ as $t\to0$, (\ref{SumForm}) may be written 
\begin{equation}
 -\sum_{i=1}^{t_n}\mathbb{P}(M(A_i) > x_n) + \sum_{i=1}^{t_n}o(\mathbb{P}(M(A_i) > x_n))  \to \text{log}(L).  \label{SumForm2}
\end{equation}
Now it is easily seen that the second sum in $(\ref{SumForm2})$ converges to zero since 
\begin{align*}
\sum_{i=1}^{t_n} o(\mathbb{P}(M(A_i) > x_n)) = \sum_{i=1}^{t_n}o\big(s_n\bar{F}(x_n)\big) & \leq t_n\,s_n\bar{F}(x_n)o(1) \\
& \leq\frac{s_n}{s_n + q_n}n\bar{F}(x_n)o(1) \to 0,
\end{align*}
and so (\ref{SumForm2}) implies $\sum_{i=1}^{t_n}\mathbb{P}(M(A_i) > x_n) \to -\text{log}(L).$  Now, decomposing the event $\{M(A_i) > x_n\}$ as a disjoint union we get 
\begin{equation}
\sum_{i=1}^{t_n} \mathbb{P}(M^i_{0,s_n-p_n} > x_n, M^i_{s_n-p_n,s_n} \leq x_n) + \sum_{i=1}^{t_n}  \mathbb{P}(M^i_{s_n-p_n,s_n} > x_n) \to -\text{log}(L). \label{SumForm3}
\end{equation}
Again, the second sum in (\ref{SumForm3}) goes to zero since it is bounded above by  $t_n p_n\bar{F}(x_n) \leq 
\{p_n/(s_n+q_n)\}n\bar{F}(x_n) \to 0,$ from which (\ref{NewCond2}) follows. 
\end{proof}

\subsection{Proof of Lemma \ref{SepLem}.}

We use induction on the number of subintervals, $k$.  The case $k=2$ is just the fact that $\{X_n\}_{n=1}^{\infty}$ satisfies AIM($x_n$).  Assuming that the result is true for $k$ such arbitrary subintervals,
 we will verify it also holds for the $k+1$ intervals $I_1, I_2,\ldots,I_k, I_{k+1}$. 
Let $I_1^*$ be the interval separating $I_1$ and $I_2$ and let $J= I_1 \cup I_1^* \cup I_2$, and we note that $J$ is an interval with $|J| > q_n$
and since $\{(M(J \cup  \cup_{i=3}^{k+1}I_i) \leq x_n \} \subseteq \{M(\cup_{i=1}^{k+1}I_i)\leq x_n)\}$ we have,
\begin{equation}
0 \leq \mathbb{P}(M(\cup_{i=1}^{k+1}I_i) \leq x_n)  - \mathbb{P}(M\big(J \cup\cup_{i=3}^{k+1}I_i\big) \leq x_n) \leq \mathbb{P}(M(I_1^*) > x_n) \leq q_n\bar{F}(x_n),
\end{equation}
so we may write $\mathbb{P}(M(\cup_{i=1}^{k+1}I_i) \leq x_n) =  \mathbb{P}(M\big(J \cup\cup_{i=3}^{k+1}I_i\big) \leq x_n) + R_{1,n}$ where the remainder $R_{1,n}$ satisfies 
$R_{1,n} \leq q_n\bar{F}(x_n)$. We then have 
\begin{align}
& \big|\mathbb{P}(M(\cup_{i=1}^{k+1}I_i) \leq x_n) - \prod_{i=1}^{k+1} \mathbb{P}(M(I_i) \leq x_n) \big|  \label{ref0}\\
=\,\, & \big|\mathbb{P}(M\big(J \cup\cup_{i=3}^{k+1}I_i\big) \leq x_n) - \prod_{i=1}^{k+1} \mathbb{P}(M(I_i) \leq x_n) + R_{1,n} \big| \\
=\,\, & \big|\mathbb{P}(M(J)\leq x_n)\prod_{i=3}^{k+1} \mathbb{P}(M(I_i) \leq x_n) - \prod_{i=1}^{k+1} \mathbb{P}(M(I_i) \leq x_n) +  R_{1,n} +  R_{2,n} \big| \label{ref1}    \\
\leq \,\, & \big|\mathbb{P}(M(J)\leq x_n) - \mathbb{P}(M(I_1)\leq x_n)\mathbb{P}(M(I_2)\leq x_n)\big| + |R_{1,n}| + |R_{2,n}|   \label{ref2}
\intertext{where $|R_{2,n}| \leq (k-1)\alpha_n + 2(k-2)q_n\bar{F}(x_n)$ and we have used our induction hypothesis to get (\ref{ref1}) since $J \cup\cup_{i=3}^{k+1}I_i$ is a union of $k$
intervals with adjacent intervals separated by $q_n$.  Now note that since $\{M(J)\leq x_n\} \subseteq \{M(I_1\cup I_2)\leq x_n\}$ we have $0 \leq \mathbb{P}(M(I_1\cup I_2)\leq x_n) - \mathbb{P}(M(J)\leq x_n) \leq q_n\bar{F}(x_n)$ and we may write $\mathbb{P}(M(J)\leq x_n) = \mathbb{P}(M(I_1\cup I_2)\leq x_n) + R_{3,n}$ where $|R_{3,n}| \leq q_n\bar{F}(x_n)$.  Thus from 
(\ref{ref0}) and (\ref{ref2}) we have}
& \big|\mathbb{P}(M(\cup_{i=1}^{k+1}I_i) \leq x_n) - \prod_{i=1}^{k+1} \mathbb{P}(M(I_i) \leq x_n) \big|  \notag \\
\leq \,\, & \big|\mathbb{P}(M(I_1\cup I_2)\leq x_n) - \mathbb{P}(M(I_1)\leq x_n)\mathbb{P}(M(I_2)\leq x_n) + R_{3,n} \big| + |R_{1,n}| + |R_{2,n}|  \notag \\
\leq \,\, & \big|\mathbb{P}(M(I_1\cup I_2)\leq x_n) - \mathbb{P}(M(I_1)\leq x_n)\mathbb{P}(M(I_2)\leq x_n) \big| + |R_{1,n}| + |R_{2,n}| + |R_{3,n}|  \notag \\ 
\leq \,\, & \alpha_n + q_n\bar{F}(x_n) + (k-1)\alpha_n + 2(k-2)q_n\bar{F}(x_n) + q_n\bar{F}(x_n) \notag \\
= \,\, & k\alpha_n + 2(k-1)q_n\bar{F}(x_n)  \notag
\end{align}
as required. 

\subsection{Proof of Lemma \ref{lemOB}.}
As $\{M_n \leq x_n\} \subseteq \cap_{i=1}^{r_n}\{M(A_i)\leq x_n\}$ we have
\begin{align}
0 \leq \mathbb{P}(M(\cup_{i=1}^{r_n}A_i) \leq x_n) - \mathbb{P}(M_n \leq x_n) & \leq \mathbb{P}(M(\cup_{i=1}^{r_n}A_i^*) > x_n)  \notag \\
& \leq r_n q_n \bar{F}(x_n)  \notag  \\
& \leq n(p_n+q_n)^{-1}q_n\{\tau\ n^{-1}+o(n^{-1})\} \to 0.  \label{OB1}
\end{align}
Also, by Lemma \ref{SepLem} we have 
\begin{equation}
\big|\mathbb{P}(M(\cup_{i=1}^{r_n}A_i) \leq x_n) - \prod_{i=1}^{r_n} \mathbb{P}(M(A_i) \leq x_n) \big| \leq (r_n-1)\alpha_n + 2(r_n-2)q_n\bar{F}(x_n) \to 0 \label{OB2}
\end{equation}
so that the triangle inequality and (\ref{OB1}) and (\ref{OB2}) give the result. 

\subsection{Proof of Theorem \ref{MainThm}.}

In addition to the notation defined in Section \ref{Sec2.1}, we also define
\begin{equation}
	  X^i_j = X_{(i-1)\,(p_n+q_n)+j}, \,\,\, M^i_{j,k} = \text{max}\,\{X^i_{j+1},\ldots, X^i_{k} \}.  \label{ExtraNotation} 
\end{equation}
Now, for $i = 1,\ldots,r_n,$ we have
\begin{align*}
  \mathbb{P}(M(A_i) \leq x_n) &= 1 - \mathbb{P}(M(A_i) > x_n) \\
                           &= 1 - \sum_{j=1}^{p_n} \mathbb{P}(X^i_j > x_n, M^i_{j,p_n} \leq x_n) \\
                           &\leq 1 - \sum_{j=1}^{p_n} \mathbb{P}(X^i_j > x_n, M^i_{j,j+p_n} \leq x_n) \\
                           &\leq \text{exp}\bigg\{-\sum_{j=1}^{p_n} \mathbb{P}(X^i_j > x_n, M^i_{j,j+p_n} \leq x_n)\bigg\}
\end{align*}
and so 
\begin{align}
  \mathbb{P}(M_n \leq x_n) &= \prod_{i=1}^{r_n} \mathbb{P}(M(A_i) \leq x_n) + o(1) \notag \\
                           %&\leq \prod_{i=1}^{r}\text{exp}\bigg\{-\sum_{j=1}^{p} \mathbb{P}(X^i_j > x_n, M^i_{j,j+p} \leq x_n)\bigg\} + o(1)  \notag \\
                           &\leq \text{exp}\bigg\{-\sum_{i=1}^{r_n}\sum_{j=1}^{p_n} \mathbb{P}(X^i_j > x_n, M^i_{j,j+p_n} \leq x_n) \bigg\}  + o(1).  \label{eq6.12}
\end{align}
Now we note that
\begin{equation}
\sum_{j=1}^{n} \mathbb{P}(X_j> x_n, M_{j,j+p_n} \leq x_n) =  \sum_{i=1}^{r_n}\sum_{j=1}^{p_n}\mathbb{P}(X^i_j > x_n, M^i_{j,j+p_n} \leq x_n)  + o(1) \label{eq6.13}
\end{equation}
since the difference between the two sums is 
\begin{align*} 
  \sum_{j=1}^{n} \mathbb{P}(X_j> x_n, M_{j,j+p_n} \leq x_n) & -  \sum_{i=1}^{r_n}\sum_{j=1}^{p_n}\mathbb{P}(X^i_j > x_n, M^i_{j,j+p_n} \leq x_n) \\
  &= \sum_{i=1}^{r_n}\sum_{j=p_n+1}^{p_n+q_n} \mathbb{P}(X^i_j > x_n, M_{j,j+p_n} \leq x_n) + o(1) \\
  & \leq r_n\,q_n\,\bar{F}(x_n)  + o(1) \\
  &\leq \frac{q_n}{p_n+q_n}\, n\,\bar{F}(x_n) + o(1) \rightarrow 0
\end{align*}
so that (\ref{eq6.12}) gives
\begin{equation}  \label{eq6.14}
  \mathbb{P}(M_n \leq x_n)  \leq  \text{exp}\bigg\{-\sum_{j=1}^{n} \mathbb{P}(X_j > x_n, M_{j,j+p_n} \leq x_n) \bigg\} + o(1).
\end{equation}

Now we prove the reverse inequality of (\ref{eq6.14}), i.e.,
\begin{equation}  \label{RevIneq}
  \mathbb{P}(M_n \leq x_n)  \geq \text{exp}\bigg\{-\sum_{j=1}^{n} \mathbb{P}(X_j > x_n, M_{j,j+p_n} \leq x_n) \bigg\} + o(1).
\end{equation}
We will show that (\ref{RevIneq}) holds under the assumption that $\mathbb{P}(M_n \leq x_n)$ converges to some $L$ in $[0,1]$, with the more general case following by considering subsequences along which $\mathbb{P}(M_n \leq x_n)$ converges and repeating the following argument.
By Lemma \ref{lemOB}, specifically equation (\ref{eq2.5}), and Lemma \ref{L5.2} with $a_{i,n} = \mathbb{P}(M(A_i) > x_n)$ we see that $L > 0$, and since (\ref{RevIneq}) trivially holds when $L=1$ we may assume $L\in(0,1)$. 
Following \cite{O'Brien87}, introduce a new sequence
$s_n = o(n)$ which plays the role of $p_n$ such that $p_n = o(s_n)$
and let $t_n = \floor*{n/(s_n + q_n)}$ which now plays the
role of $r_n$ and note that $t_n = o(r_n)$ and the definitions in
(\ref{ExtraNotation}) and (\ref{A_i}) are modified by replacing $p_n$
with $s_n$.  
 Then for $i = 1,\ldots t_n$, we have 
\begin{align*}
  \mathbb{P}(M(A_i) > x_n) &= \mathbb{P}(M^i_{0,s_n-p_n} > x_n, M^i_{s_n-p_n,s_n} \leq x_n) + \mathbb{P}(M^i_{s_n-p_n,s_n} > x_n) 
\end{align*}
and consequently, since Lemma \ref{lemOB} holds with $s_n$ in place of $p_n$ and $t_n$ in place of $r_n,$ 
\begin{align}
  \mathbb{P}(M_n \leq x_n) &= \prod_{i=1}^{t_n} \mathbb{P}(M(A_i)\leq x_n) + o(1) =
 \prod_{i=1}^{t_n}\bigg\{1 - \mathbb{P}(M(A_i) > x_n)\bigg\} + o(1)  \notag \\
&= \prod_{i=1}^{t_n}\bigg\{1 - \mathbb{P}(M^i_{0,s_n-p_n}>x_n, M^i_{s_n-p_n,s_n}\leq x_n) - \mathbb{P}(M^i_{s_n-p_n, s_n} > x_n)\bigg\} + o(1).  \label{ProdForm}
\end{align}
Now, with $a_{i,n} = \mathbb{P}(M^i_{0,s_n-p_n}>x_n, M^i_{s_n-p_n,s_n}\leq x_n) + \mathbb{P}(M^i_{s_n-p_n, s_n} > x_n) $ we have,
for all $i$, $a_{i,n} \leq s_n\bar{F}(x_n)$ and so $\text{max}_{1\leq i \leq t_n}a_{i,n} \leq s_n\bar{F}(x_n) \to 0$ as $s_n=o(n)$ and $\bar{F}(x_n) = \tau\,n^{-1} + o(n^{-1})$. Also, $\sum_{i=1}^{t_n}a_{i,n} \leq t_n\text{max}_{1\leq i \leq t_n}a_{i,n} \leq s_n(s_n+q_n)^{-1}n\bar{F}(x_n)$ which is bounded above as $n\to\infty$.
Thus we may apply Lemma \ref{L5.2} to (\ref{ProdForm}) to get 
\begin{align}
\mathbb{P}(M_n \leq x_n) &= \text{exp}\bigg\{-\sum_{i=1}^{t_n}  \bigg(\mathbb{P}(M^i_{0,s_n-p_n}>x_n, M^i_{s_n-p_n,s_n}\leq x_n) + \mathbb{P}(M^i_{s_n-p_n, s_n} > x_n) \bigg)\bigg\} + o(1)  \notag  \\
&= \text{exp}\bigg\{-\bigg( \sum_{i=1}^{t_n}\mathbb{P}(M^i_{0,s_n-p_n}>x_n, M^i_{s_n-p_n,s_n}\leq x_n)\bigg)\bigg(1 + o(1) \bigg) \bigg\} + o(1) \label{ExpForm} 
\end{align}
with (\ref{ExpForm}) following from Lemma \ref{L5.4}.   Now Lemma \ref{L5.3} reduces (\ref{ExpForm}) to
\begin{align}
\mathbb{P}(M_n \leq x_n) &= \text{exp}\bigg\{-\sum_{i=1}^{t_n}\mathbb{P}(M^i_{0,s_n-p_n}>x_n, M^i_{s_n-p_n,s_n}\leq x_n) \bigg\} + o(1) \label{RevSetup}  \\
& \geq \text{exp}\bigg\{-\sum_{i=1}^{t_n}\sum_{j=1}^{s_n}\mathbb{P}(X^i_j > x_n, M^i_{j,j+p_n} \leq x_n)  \bigg\} + o(1) \label{RevIneq2}
\end{align}
with (\ref{RevIneq2}) following from (\ref{RevSetup}) by the inclusions $\{ M^i_{0, s_n-p_n} > x_n, M^i_{s_n-p_n,s_n} \leq x_n\}  \subseteq \, \bigcup_{j=1}^{s_n-p_n}\{X^i_j > x_n, M^i_{j,j+p_n} \leq x_n \} \subseteq \, \bigcup_{j=1}^{s_n}\{X^i_j > x_n, M^i_{j,j+p_n} \leq x_n \}$ and the union bound.
A similar argument that was used to show (\ref{eq6.13}) gives 
\begin{equation}
  \sum_{j=1}^{n} \mathbb{P}(X_j > x_n, M_{j,j+p_n} \leq x_n) =  \sum_{i=1}^{t_n}\sum_{j=1}^{s_n}\mathbb{P}(X^i_j > x_n, M^i_{j,j+p_n} \leq x_n)  + o(1) \label{eq6.17}
\end{equation}
so that (\ref{RevIneq2}) becomes
\begin{equation}  \label{eq6.18}
  \mathbb{P}(M_n \leq x_n)  \geq  \text{exp}\bigg\{-\sum_{j=1}^{n} \mathbb{P}(X_j > x_n, M_{j,j+p_n} \leq x_n) \bigg\} + o(1)
\end{equation}
and so  (\ref{eq6.14}) and (\ref{eq6.18}) together prove (\ref{result1}). Also, since
\begin{align*}
  \text{exp}\bigg\{-\sum_{j=1}^{n} \mathbb{P}(X_j > x_n, M_{j,j+p_n} \leq x_n) \bigg\} = \bigg[\text{exp}\big\{-n\bar{F}(x_n) \big\}\bigg]^{\gamma_n}
\end{align*}
with
$\gamma_n = n^{-1}\sum_{j=1}^{n}\mathbb{P}(M_{j,j+p_n} \leq x_n \mid X_j > x_n)$, this also gives (\ref{result2}).

\subsection{Proof of Theorem \ref{UniformConvThm}.}  \label{UniformConvThmPf}
Throughout we let $\theta_{i,n} = \mathbb{P}(M_{i,i+p_n} \leq x_n \mid X_i > x_n)$. From Corollary (\ref{EasyCor}) we know that 
$\mathbb{P}(M_n \leq x_n) \to e^{-\tau\gamma'}$ with $\gamma' = \lim_{n\to\infty}n^{-1}\sum_{i=1}^{n}\theta_{i,n}$ which is easily seen to converge to
 the same value as $\lim_{n\to\infty}n^{-1}\sum_{i=1}^{n}\theta_i$ since 
\begin{align}
| n^{-1}\sum_{i=1}^{n}\theta_{i} - n^{-1}\sum_{i=1}^{n}\theta_{i,n} | & \leq n^{-1} \sum_{i=1}^{n}|\theta_{i,n} - \theta_i |  \notag \\
& \leq \text{max}_{1\leq i \leq n} |\theta_{i,n} - \theta_i| \to 0.  \label{ToZero}
\end{align}
This establishes the first part of the Theorem.

To show that, $\mathbb{P}(M_n \leq y_n) \to e^{-\tau'\gamma}$, define $n' =  \floor*{ (\tau'/\tau)n}$ so that $n'\bar{F}(x_n) \to \tau'$ or equivalently
$F(x_n)^{n'} \to e^{-\tau'}$.  Then,
\begin{align}
| \mathbb{P}(M_{n\textquotesingle} \leq x_n) - \mathbb{P}(M_{n\textquotesingle} \leq y_{n\textquotesingle}) |  & \leq 
 n\textquotesingle | F(x_n) - F(y_{n\textquotesingle}) |   \notag  \\
& = n\textquotesingle | \bar{F}(x_n) - \bar{F}(y_{n\textquotesingle}) | \to 0.  \label{ProbDiff}
\end{align}
Since $n' \leq n$ we have by Theorem \ref{MainThm}, $\mathbb{P}(M_{n'} \leq x_n) = F(x_n)^{n'\gamma_{n}'}$ where
 $\gamma_{n}' =  (n')^{-1}\sum_{i=1}^{n'} \theta_{i,n}$ and $\gamma_{n}' $ also has limiting value $\lim_{n\to\infty}n^{-1}\sum_{i=1}^{n}\theta_i $ since
$| (n')^{-1}\sum_{i=1}^{n'} \theta_{i,n} - (n')^{-1}\sum_{i=1}^{n'} \theta_i | \to 0$ as in (\ref{ToZero}). Then, since $F(x_n)^{n'} \to e^{-\tau'}$, we have
$\mathbb{P}(M_{n'} \leq x_n) \to e^{-\tau'\gamma}$ and so from (\ref{ProbDiff}), $\mathbb{P}(M_n \leq y_n) \to  e^{-\tau'\gamma}$ also with
$\gamma$ as in (\ref{CesaroMean}). 

\subsection{Proof of Theorem \ref{PeriodicThm}.} 
The first step in the proof is to show that  we have, for each integer $k\geq 1$, 
\begin{equation}
\mathbb{P}(M_n \leq x_n) - \mathbb{P}^k(M_{n'} \leq x_n) \to 0  \label{FirstStep}
\end{equation}
where $n' = \floor*{n/k}$. To do this we define intervals $I_i$ and $I_i^*$, $1\leq i \leq k,$ of alternating large and small lengths as follows,
\begin{equation}
I_i = \{(i-1)n' + 1, \ldots, in' - q_n\}, \quad I_i^* = \{in'-q_n+1, \ldots, in'\}.
\end{equation}
We show that 
\begin{align}
&  \quad \big| \mathbb{P}(M_n \leq x_n) -  \mathbb{P}(M(\cup_{i=1}^{k}I_i) \leq x_n)\big| \to 0, \label{step1} \\
&  \quad \big| \mathbb{P}(M(\cup_{i=1}^{k}I_i) \leq x_n) - \prod_{i=1}^{k}\mathbb{P}(M(I_i)\leq x_n)\big| \to 0,  \label{step2} \\
&  \quad \big| \prod_{i=1}^{k}\mathbb{P}(M(I_i)\leq x_n) - \prod_{i=1}^{k}\mathbb{P}(M(I_i\cup I_i^*)\leq x_n) \big| \to 0,  \label{step3} \\
\intertext{and} &  \quad \big| \prod_{i=1}^{k}\mathbb{P}(M(I_i\cup I_i^*)\leq x_n) - \mathbb{P}^k(M_{n'} \leq x_n)\big| \to 0 \label{step4},
\end{align}
from which (\ref{FirstStep}) follows by the triangle inequality.

(\ref{step1}) follows from $\{M_n \leq x_n\} \subseteq \{M(\cup_{i=1}^{k}I_i) \leq x_n\}$ and so
 $\mathbb{P}(M(\cup_{i=1}^{k}I_i) \leq x_n) - \mathbb{P}(M_n \leq x_n) \leq \mathbb{P}(\cup_{i=1}^{k}\{M(I_i^*) > x_n\}) \leq kq_n\bar{F}(x_n) \to 0$ since $q_n = o(n)$ 
and $\bar{F}(x_n) = \tau/n + o(n^{-1}).$ 

(\ref{step2}) follows immediately from Lemma \ref{SepLem} as $I_j, 1\leq j \leq k,$ are distinct subintervals of $\{1,\ldots, n\}$, and $I_i$ and
$I_{i+1}$ are separated by $q_n$.  

(\ref{step3}) follows from $\{M(I_i \cup I_i^*) \leq x_n \} \subseteq \{M(I_i) \leq x_n\}$ and $0\leq \mathbb{P}(M(I_i) \leq x_n) - \mathbb{P}(M(I_i \cup I_i^*) \leq x_n) \leq q_n\bar{F}(x_n) \to 0$ so that $ \mathbb{P}(M(I_i) \leq x_n) = \mathbb{P}(M(I_i \cup I_i^*) \leq x_n) + o(1)$ for $1\leq i \leq k$. 

For (\ref{step4}), we first note that $\mathbb{P}(M(I_1\cup I_1^*)\leq x_n) = \mathbb{P}(M_n' \leq x_n).$ Since $I_i\cup I_i^*$ is an interval of length $n'$, $1\leq i\leq k$,
 we have by periodicity that
$M(I_i\cup I_i^*) \overset{D}{=} M_{r,r+n'}$ for some $r\in\{0,1,\ldots, d-1\}.$ Then for $2 \leq i \leq k$, we have
\begin{align*}
 | \mathbb{P}(M(I_i\cup I_i^*)) - \mathbb{P}(M(I_1\cup I_1^*)) | &= | \mathbb{P}(M_{r,r+n'} \leq x_n) - \mathbb{P}(M_n' \leq x_n) | \\
& \leq \,\, | \mathbb{P}(M_{r+n'} \leq x_n) -  \mathbb{P}(M_n' \leq x_n) |\,\, + \\
& \phantom{==} | \mathbb{P}(M_{r,r+n'} \leq x_n) - \mathbb{P}(M_{r+n'} \leq x_n) |  \\
& \leq \,\, r\bar{F}(x_n) + r\bar{F}(x_n) \\
& \leq 2d\bar{F}(x_n) \to 0.
\end{align*} 
Hence $\prod_{i=1}^{k}\mathbb{P}(M(I_i\cup I_i^*)\leq x_n) = \prod_{i=1}^{k}\mathbb{P}(M(I_1\cup I_1^*)\leq x_n) + o(1) =   \mathbb{P}^k(M_{n'} \leq x_n) + o(1).$
 which establishes (\ref{FirstStep}). 

Now we note that since $\{X_n\}_{n=1}^{\infty}$ satisfies AIM($x_n$) with $n\bar{F}(x_n) \to \tau$, it also satisfies AIM($y_n)$ whenever 
$n\bar{F}(y_n) \to \tau' \leq \tau$. This follows in the exact same way as for the $D(x_n)$ condition, see, e.g., Lemma 3.6.2.\ in \cite{leadling83}.  This fact together with
(\ref{FirstStep}) allows the proof to proceed in exactly the same manner as the proof of Theorem 3.7.1.\ in \cite{leadling83}.  

\subsection{Proof of Theorem \ref{IntExcThm}.} 
For $n, q \in \mathbb{N}$ and $u\in\mathbb{R}$, we define the mixing coefficients $\alpha_{n,q}(u)$ by
\begin{equation*}
\alpha_{n,q}(u) = \text{max} \mid\mathbb{P}\big(M(I_1 \cup I_2) \leq u \big) - \mathbb{P}\big(M(I_1) \leq u\big)\mathbb{P}\big(M(I_2) \leq u\big)\mid
\end{equation*}
where the maximum is taken over intervals $I_1$ and $I_2$ that are separated by $q$, with $\text{min}\{|I_1|, |I_2|\} \geq q$, $\text{min}\{\text{min}(I_1), \text{min}(I_2)\} \geq 1$ and max$\{\text{max}(I_1), \text{max}(I_2)\} \leq n$.

We first note that since both $q_n = o(n)$ and $0 \leq \alpha_n = \alpha_{n,q_n}(x_n) \leq \alpha^*_{n,q_n}(x_n)   \to 0$, we can find a sequence of positive integers $p_n = o(n)$ such that $n\alpha_n = o(p_n)$ and 
$q_n = o(p_n)$ so that the conditions of Theorem \ref{MainThm} are satisfied.

Let $t>0$ and write $k_n = \floor*{t/\bar{F}(x_n)} \sim tn/\tau$ so that for sufficiently large $n,$ $k_n > p_n + q_n$.  
Now, fix $i\in \mathbb{N}$.  For sufficiently large $n$ we have
\begin{align*}
& \mathbb{P}(M_{i+p_n,i+p_n+q_n} > x_n \mid X_i > x_n) \leq q_n\bar{F}(x_n) + \alpha^*_{n,q_n}(x_n)  \rightarrow 0  \notag
\intertext{so that}
& \mathbb{P}(M_{i,i+k_n} \leq x_n \mid X_i > x_n) = \mathbb{P}(M_{i,i+p_n} \leq x_n, M_{i+p_n+q_n,i+k_n} \leq x_n \mid X_i > x_n)  + o(1).  
\end{align*}
In a similar way, since $\{ M_{i+k_n} \leq x_n \}  \subseteq \{ M_{i+p_n+q_n, i + k_n} \leq x_n \},$ we have
\begin{align*}
\mathbb{P}( M_{i+p_n+q_n, i + k_n} \leq x_n ) - \mathbb{P}(M_{i+k_n} \leq x_n )  &= \mathbb{P}(M_{i+p_n+q_n} > x_n, M_{i+p_n+q_n,i+k_n} \leq x_n)  \\
& \leq (i+p_n+q_n)\bar{F}(x_n) \rightarrow 0,
\end{align*}
so that $\mathbb{P}(M_{i+p_n+q_n,i+k_n} \leq x_n) = \mathbb{P}(M_{i+k_n} \leq x_n) + o(1).$  Now we can derive the limiting distribution of $\bar{F}(x_n)T_i(x_n)$.  We have
\begin{align}
  \mathbb{P}(\bar{F}(x_n)T_i(x_n) > t) &= \mathbb{P}(T_i(x_n) > k_n)  = \mathbb{P}( X_{i+1} \leq x_n,\ldots X_{i+k_n} \leq x_n  \mid X_i > x_n)  \notag \\
                                       & =\mathbb{P}(M_{i,i+k_n} \leq x_n \mid X_i > x_n)    \notag \\    
                                       & =  \mathbb{P}(M_{i,i+p_n} \leq x_n, M_{i+p_n+q_n,i+k_n} \leq x_n \mid X_i > x_n)  + o(1) \notag \\    
                                       & = \mathbb{P}(M_{i,i+p_n} \leq x_n \mid X_i> x_n)\mathbb{P}(M_{i+p_n+q_n,i+k_n} \leq x_n \mid X_i > x_n, M_{i,i+p_n} \leq x_n)+o(1)   \notag \\   
                                       & = \big\{\theta_i + o(1) \big\}\big\{\mathbb{P}(M_{i+k_n} \leq x_n) + o(1) \big\} + o(1).   \label{LastEqn}
\end{align}
Now we focus on the term $\mathbb{P}(M_{i+k_n} \leq x_n)$ appearing in (\ref{LastEqn}).  Since $\mathbb{P}(M_{k_n} \leq x_n) - \mathbb{P}(M_{i+k_n} \leq x_n) \leq i\bar{F}(x_n)$
we have $\mathbb{P}(M_{i+k_n} \leq x_n) = \mathbb{P}(M_{k_n} \leq x_n)+ o(1).$  Since $k_n = O(n)$ we have $\alpha^*_{k_n,q_n}(x_n) \to 0$ and so
$\alpha_{k_n,q_n}(x_n) \to 0$ also.  Thus we may find a sequence $p_n^{\prime} = o(n)$ such that $k_n\alpha_{k_n,q_n} = o(p_n^{\prime})$ and $q_n = o(p_n^{\prime})$, 
e.g., we may take $p_n^{\prime} =   \floor*{(k_n\, \text{max}\{q_n, k_n\alpha_{k_n,q_n}(x_n)\})^{1/2}}.$  Then applying Theorem \ref{MainThm} to the first $k_n$ terms we get
$\mathbb{P}(M_{k_n} \leq x_n) = F(x_n)^{k_n\gamma_n^{\prime}}$ where $\gamma_n^{\prime} = k_n^{-1}\sum_{j=1}^{k_n}\theta_{j,n}^{\prime}$ with
$\theta_{j,n}^{\prime} = \mathbb{P}(M_{j,j+p_n^{\prime}} \leq x_n \mid X_j > x_n)$.

We now verify that $\gamma_n^{\prime}$ has limiting value $\gamma = \lim_{n\to\infty}n^{-1}\sum_{j=1}^{n}\theta_j$. Define sequences $a_n$, $b_n$ and $k_n^{\prime}$ by
$a_n = \text{max}\{p_n, p_n^{\prime}\}, b_n = \text{min}\{p_n, p_n^{\prime}\}$ and $k_n^{\prime} = k_n + a_n$ and note that $k_n^{\prime} = O(n)$. 
Then with the usual notation $\theta_{j,n} = \mathbb{P}(M_{j,j+p_n} \leq x_n \mid X_j > x_n)$, we have for all $1\leq j \leq k_n$ that, for sufficiently large $n$,
$|\theta_{j,n}^{\prime} - \theta_{j,n}|  \leq \, \mathbb{P}(M_{j+b_n, j+a_n} > x_n \mid X_j > x_n)  \leq |p_n - p_n^{\prime}|\,\bar{F}(x_n) + \alpha^*_{k_n^{\prime},q_n}(x_n)$
where we have used the fact that $b_n > q_n$ for sufficiently large $n$ and $\alpha^*_{n,q}(u)$ is non-decreasing in $n$ for fixed $q$ and $u$.
Therefore, $|k_n^{-1}\sum_{j=1}^{k_n}\theta_{j,n} - k_n^{-1}\sum_{j=1}^{k_n}\theta_{j,n}^{\prime}| \leq |p_n - p_n^{\prime}|\,\bar{F}(x_n) + \alpha^*_{k_n^{\prime},q_n}(x_n)\to 0$ and so $k_n^{-1}\sum_{j=1}^{k_n}\theta_{j,n} $ and $k_n^{-1}\sum_{j=1}^{k_n}\theta_{j,n}^{\prime}$ have the same limit which equals $\gamma$ since
$|k_n^{-1}\sum_{j=1}^{k_n}\theta_{j,n} - k_n^{-1}\sum_{j=1}^{k_n}\theta_j| \leq \text{max}_{1\leq j\leq k_n}|\theta_j - \theta_{j,n}| \to 0.$

Finally, we have $\mathbb{P}(M_{i+k_n} \leq x_n) = \mathbb{P}(M_{k_n} \leq x_n)+ o(1) = F(x_n)^{k_n(\gamma + o(1))} = [e^{-t} + o(1)]^{\gamma + o(1)}$
since $n\bar{F}(x_n)\to \tau$ implies that $k_n\bar{F}(x_n)\to t$ which in turn implies that $F(x_n)^{k_n}\to e^{-t}$. 
Substituting $\mathbb{P}(M_{i+k_n} \leq x_n) = e^{-\gamma\,t} + o(1)$ in (\ref{LastEqn}) then gives the result.

\vspace{5mm} 
\textbf{Acknowledgements.} The authors would like to express their gratitude to an anonymous reviewer whose helpful comments have greatly improved this paper. 

\vspace{5mm} 
\textbf{Data availability.} The datasets generated and analysed during the current study are available from the corresponding author upon request.

% {\small
% \bibliographystyle{agsm}
% \bibliography{bibliography2}
% }

\end{document}